\documentclass{article}

\usepackage{amsfonts}
\usepackage{amsthm}
\usepackage{amssymb}
\usepackage{soul}

\date{}

\usepackage{amsmath, amsgen, color}
\usepackage[english]{babel}

   \newcommand{\beq}{\begin{equation}}
   \newcommand{\eeq}{\end{equation}}
   \newcommand{\beqs}{\arraycolsep1.5pt\begin{eqnarray}}
   \newcommand{\eeqs}{\end{eqnarray}\arraycolsep5pt}
   \newcommand{\beqsn}{\arraycolsep1.5pt\begin{eqnarray*}}
   \newcommand{\eeqsn}{\end{eqnarray*}\arraycolsep5pt}

\newtheorem{pft}{}                              

\newcommand{\bpf}{\begin{pft}\begin{proof}\mbox{\bf Proof}.\ \rm}
\newcommand{\bpfof}[1]{\begin{pft}\begin{proof}{\bf Proof of #1}\
\rm}
\newcommand{\bspf}{\begin{pft}\begin{proof}{\bf Sketch of proof}.\
\rm}

\newcommand{\epf}{\nopagebreak\end{proof}\end{pft}}

\newtheorem{thm}{Theorem}[section]
\newtheorem{rem}[thm]{Remark}
\newtheorem{cor}[thm]{Corollary}
\newtheorem{prop}[thm]{Proposition}

\newtheorem{defn}[thm]{Definition}

\newcommand{\dis}{\displaystyle}

\def\Bbb#1{{\fam\msbfam\relax#1}}

\def\R{{\mathbb R}}
\def\N{{\Bbb N}}

\newcommand{\ds}{\rightarrow}

\newcommand{\ee}{\varepsilon}

\newcommand{\la}{\lambda}

\newcommand{\Om}{\Omega}
\newcommand{\f}{\varphi}

\def\B{{\cal B}}

\def\F{{\cal F}}

\def\L{{\cal L}}

\newcommand{\ninf}{_{n \ds \infty}}

\def\supess{\mathop{\rm ess\: sup }}

\newcommand{\dy}{\,dy}

\def\supess{\mathop{\rm ess\: sup }}

\font\tenmsb=msbm10 \font\sevenmsb=msbm7 \font\fivemsb=msbm5
\newfam\msbfam
\textfont\msbfam=\tenmsb \scriptfont\msbfam=\sevenmsb
\scriptscriptfont\msbfam=\fivemsb
\def\Bbb#1{{\fam\msbfam\relax#1}}

\numberwithin{equation}{section}

\def\to{\rightarrow}

\title{A relaxation result in the vectorial setting and $L^p$-approximation  for $L^\infty$-functionals.}

\author{{\sc Francesca Prinari}\\
	Dip. di Matematica e Informatica \\ Universit\`a di Ferrara \\ Via Machiavelli 35\\
	44121 Ferrara (Italy) 
	\\
	and\\
	{\sc Elvira Zappale} \\
Dip. di Ingegneria Industriale 
\\
Universit\`a degli Studi di Salerno\\
Via Giovanni Paolo II, 132\\
84084 Fisciano (SA) (Italy)}

\begin{document}
\maketitle
\begin{abstract}
	We provide relaxation for not lower semicontinuous supremal functionals of the type $W^{1,\infty}(\Omega;\mathbb R^d) \ni u \mapsto\supess_{ x \in \Omega}f(\nabla u(x))$ in the vectorial case, where $\Omega\subset \mathbb R^N$ is a Lipschitz, bounded open set, and $f$ is level convex. The connection with indicator functionals is also enlightened, thus extending previous lower semicontinuity results in that framework. Finally we discuss the $L^p$-approximation of supremal functionals, with  non-negative, coercive densities $f=f(x,\xi)$, which are only $\L^N \otimes \B_{d \times N}$-measurable.  
\end{abstract}
\bigskip\noindent{{\it Keywords: supremal functionals, relaxation, level convexity, $\Gamma$-convergence.} 

}

\section{Introduction}\label{intro}

Recently a great attention has been devoted to supremal functionals, i.e. functionals of the type 
\begin{align}\label{supess}
W^{1,\infty}(\Omega;\mathbb R^{d})\ni u\to F(u):=\supess_{ x \in \Omega}f(\nabla u(x)),
\end{align}
$\Omega$ being a bounded open set of $\R^N$ with Lipschitz boundary, and to their connections with partial differential equatons such as $\infty$-harmonic, $\infty$-biharmonic equations or Hamilton-Jacobi ones, also in light of the many applications to optimal transport, continuum mechanics, see for instance \cite{ABM, BBJ0, BBJ, BM1, BGP, K, KP, KPr, KS} among a wider literature. 
Many of the above questions can be formulated in terms of suitable minimization problems involving \eqref{supess}, and the direct methods have been proven to be a powerful tool to provide solutions. 
A crucial property to ensure the existence of minimizers is the lower semicontinuity of the functional $F$ in \eqref{supess}.
This in turn reflects in necessary and sufficient conditions on the supremand $f$. Such analysis started in the scalar case $d$ or $N=1$   in  \cite{BJ} and \cite{ABP}, and later extended in \cite{BJW, P08, P09} and so far a complete characterization is given: $F$ is weakly* sequentially lower semicontinuous if and only if  $f$ is lower semicontinuous and level convex, i.e. its sublevel sets \begin{equation}\label{Ela}
L_\lambda(f):=\{\xi \in \R^{d\times N}: f(\xi)\leq \lambda\},
\end{equation} 
are closed and convex. 

When the problem is truly vectorial, lower semicontinuity and  level convexity of the supremand $f$ are just sufficient conditions but no longer necessary. The notion which has been proven to be necessary and sufficient for weak*  sequential lower semicontinuity of $F$  in $W^{1,\infty}(\Omega;\mathbb R^{d\times N})$ is
{\it strong Morrey quasiconvexity}, introduced by \cite{BJW}, which reads as follows.
A Borel measurable function $f : \R^{d \times N}\to  \R$ is said to be
{\it strong Morrey quasiconvex} if for any $\varepsilon >0$, for any $\xi \in \R^{d \times N}$, and for any $K > 0$,
there exists a $\delta = \delta(\varepsilon,K, \xi) > 0$ such that if $\varphi \in W^{1,\infty}(Q;\R^d)$ satisfies
$
\|\nabla \varphi\|_{L^\infty(Q)}\leq K, \hbox{ and }
\max_{x \in \partial Q}
|\varphi (x)| \leq \delta,$
then
\begin{equation}\label{sMqcx}
f(\xi) \leq \supess_{x\in Q}f(\xi + \nabla \varphi(x)) + \varepsilon,
\end{equation}
where $Q$ denotes the cube $]0,1[^N$.
This notion is quite difficult to be verified in practice and stronger notions (but weaker than level convexity) have been introduced in order to ensure lower semicontinuity to supremal functionals or approximate them through integral functionals (see \cite{AP0, BJW, CDPP}). 

 Clearly if such conditions fail to be satisfied by the supremand $f$, one has to look for the best weak* lower semicontinuous functional $\Gamma_{w^*}(F)$, which approximates $F$ in the sense of admitting the same minimal values (see Theorem \ref{convmin} below). 
 The results available in literature are very satisfactory and quite exhaustive  in the case $d=1$  or $N =1$ when $F$ satisfies a coercivity assumption:
 in this case, $\Gamma_{w^*}(F)=\Gamma_{w^*_{seq}}(F)$ where $w^*_{seq}$ is the weak* sequential  topology   on $W^{1,\infty}(\Omega;\R^d)$.
 In \cite{BL} and \cite{P09} a complete representation formula for  the relaxed functional $\Gamma_{w^*}(F)$ 
  is given when $f=f(x,\xi)$ is a globally continuous function; in \cite{GPP}  the authors discuss the finslerian case and represent  $\Gamma_{w^*}(F)$   as a difference quotient; in  \cite{P09}  it is shown that  $\Gamma_{w*}(F)$ is level convex
 when $f$ is a  Carath\'eodory function satisfying  the additional assumption $f(x,\xi)=f(x,-\xi)$. In \cite{GP} the last assumption is dropped  and  the level convexity of the relaxed functional  is proved  for  a class of discontinuos supremand, not even coercive.
 Despite of all these scalar results,  very little is known in the vectorial setting, up to some sufficient conditions and in  particular cases (see \cite{AP, AP0, BJW, CDPP}).
 
 The first aim of this paper consists  in providing a relaxation result for a  class of supremal functionals when $d$ and $N>1$.  In the vectorial case we compute the lower semicontinuous envelope of $F$ in \eqref{supess}   with respect to the weak* topology when the supremand $f$ is level convex and  only  Borel measurable. We remark that level convex functions are not lower semicontinuous by definition (see \cite{RZ}).
 Our main result (which clearly holds also in the scalar case), is  the following:

 \begin{thm}\label{relax1}Let $\Omega$ be a bounded  open set of $\R^N$ with Lipschitz boundary and let $f: \R^{d\times N}\to \overline\R$ be a  Borel function such that \\
 
 $(H)$ for every $\la> \inf_{\mathbb R^{d \times N}}f$ the sublevel set $L_{\la}(f)$ in \eqref{Ela} is convex and  has nonempty interior.\\

\noindent Let $F:W^{1,\infty}(\Omega;\R^d)\to{\overline \R} $ be the supremal functional in \eqref{supess}. 
 	Then it holds
 	\beq\label{relrep}\Gamma_{w*}(F)(u)=\Gamma_{w^*_{seq}}(F)(u)=\supess_{x \in \Omega} f^{ls}(\nabla u(x)) \quad  \hbox{ for every  } u\in W^{1,\infty}(\Omega;\R^d)
	\eeq
	
 	 where $f^{ls}$ denotes the lower semicontinuous envelope of $f$.
 \end{thm}

Note that $(H)$  is satisfied by a  wide class of discontinuous functions. For instance, it is satisfied by Borel level convex functions $f$  having an absolute minimum point $\bar \xi$ such that
$
 f(\overline{\xi})= \lim_{\xi\to \overline{ \xi}}f(\xi).
$

 Moreover in Corollary \ref{corfls} we show that if $f$ satisfies $(H)$ then $f^{ls}$  is the greatest strong Morrey quasiconvex function less  than or equal to  $f$.
 

%
%
%
%
%

Note that the class of Borel level convex functions is strictly contained in the class  of Borel functions $f$ (called {\sl weak Morrey quasiconvex}) satisfying 
$$
f(\xi) \leq \supess_{x\in Q}f(\xi + \nabla \varphi(x)),\quad \forall  \varphi \in W^{1,\infty}_0(Q;\R^d).
$$
 Differently from  strong Morrey quasiconvex functions which are lower semicontinuous (see \cite[Proposition 2.5]{BJW}), weak Morrey quasiconvex functions do not necessarily satisfy  this property.
  On the other hand, the representation result by means of $f^{ls}$ does not hold if we weaken the level convexity assumption on $f$, by  requiring that $f$ is 
only weak Morrey quasiconvex. Indeed, \cite[Example 2.7]{PNodea} exhibits  a weak Morrey quasiconvex function $f=f^{ls}$ that cannot represent the relaxed functional since it is not strong Morrey quasiconvex.\\

In order to prove Theorem \ref{relax1}, a key tool is the description of the level sets of the envelopes of the densities $f$, that is accomplished in  
Section \ref{prelim}. Indeed, after providing in Proposition \ref{flcbuild} a characterization of the level convex envelope of functions defined in generical vector spaces $(X,\tau)$, we specialize the result, giving  a complete representation formula of the sublevel sets of  $f^{lslc}$ in terms of closures and convexifications  of the sublevel sets of $f$ (see Proposition \ref{properties}). For computational counterpart in the continuous and bounded case we refer to \cite{AO} while in the nonlocal setting formulas analogous to \eqref{Flcrep} can be found in \cite{KZ}.\\

We also underline that, despite of the results currently available in the literature,  in the set of hypothesis of Theorem \ref{relax1}  we  drop any coercivity assumptions on $f$ thanks to arguments as in \cite[Theorem 3.1]{GP}.
 On the other hand,  the proof of  representation formula \eqref{relrep}  is given under homogeneity assumptions on the density $f$ since it relies on a particular case of \cite[Theorem 2.1]{AHM} (see Theorem \ref{AHM} below). Indeed a central role plays the connection with homogeneous indicator functionals of bounded convex sets with nonempty interior, as already emphasized in similar context by \cite{BJW} and later exploited in \cite{BGP}, and very recently in \cite{KZ, KZ2} in the nonlocal framework. 
 In turn, Theorem \ref{relax1} allows us to generalize some relaxation results for indicator functionals or, equivalently, improves the understanding of the asymptotics for vectorial differential inclusions (cf. Corollary \ref{unbddint} below). The interest in this type of functionals is motivated by the many applications: we refer to \cite{CDeA} and the references therein for the scalar case, to \cite{W0, W1} for multidimensional control problems,  to \cite{DG} for homogenization, to \cite{BPZ, Z0, Z} for the analysis of thin structures, and to \cite{BKS}, and the bibliography contained therein for the applications in continuum mechanics.

Motivated by the connection with PDEs and norm approximation, the last section of our paper is devoted to  an $L^p$-approximation theorem which  applies to a more general class of densities $f=f(x,\xi)$. Our result generalizes  \cite[Theorem 3.2]{PNodea}, since, under the same growth conditions, we just require measurability for $f$.


%
%
%
%
%

\begin{thm}\label{curlcase2} Let $\Omega\subseteq \R^N$ be a bounded open set with Lipschitz boundary.
	Let $f:\Omega\times \R^{d\times N}\to[0,+\infty)$ be a $\L^N \otimes \B_{d\times N}$-measurable function satisfying the following  growth condition:
	there exist $\beta\geq\alpha>0$ such that
	\beq\label{fcrescita2} 
	\frac{1}{\alpha}|\xi|-\alpha\leq f(x,\xi) \leq \beta(1+|\xi|) \qquad \hbox{ for a. e. }x \in \Omega \hbox{ and for every }
	\xi\in \R^{d\times N}.
	\eeq
	For every $p\geq 1$ let  
	$F_p: C(\bar \Omega;\R^d)\to [0,+\infty)$  be the functional given by 
	\beq\label{curl2} F_p(u):=\left\{\begin {array}{cl}
	\displaystyle \left( \int_{\Om} f^p(x,\nabla u(x))dx \right)^{1/p}
	&  \hbox{if } \, u\in W^{1,p}(\Omega;\R^d),\\
	+\infty  & \hbox{otherwise}.
\end{array}\right.
\eeq
\noindent Then there exists a $\L^N \otimes \B_{d\times N}$-measurable function $f_{\infty}:\Omega\times  \R^{d\times N}\to [0,+\infty)$ such that    $(F_p)_{p\geq 1}$ $\Gamma(L^{\infty})$-converges, as $p\to\infty$, to the functional $\bar F:C(\bar \Omega;\R^d)\to \overline \R$ defined as

 \begin{equation}\label{Finfty2}
\bar F(u):=\left\{\begin {array}{cl} \displaystyle \supess_{x\in \Om}
f_{\infty}(x,\nabla u(x))
&  \hbox{if } \, u\in W^{1,\infty}(\Omega;\R^d),\\
+\infty  & \hbox{otherwise.}
\end{array}\right.
\end{equation}
Moreover for a.e. $x\in \Omega$ $f_{\infty}(x,\cdot)$  is a strong Morrey quasiconvex function satisfying 
\beq\label{comparison}f_{\infty}(x,\cdot)\geq Q_{\infty}f(x,\cdot):=\sup_{p\geq 1} (Qf^p)^{1/ p}(x,\cdot),
\eeq
 where $Qf^p(x,\cdot):=Q(f^p)(x,\cdot)$ stands for the quasiconvex envelope of $f^p(x,\cdot)$ (cf. \eqref{Qg}).

\end{thm}

In particular, under the assumptions of Theorem \ref{relax1}, the latter result guarantees that
the relaxed functional $W^{1,\infty}(\Omega;\mathbb R^d) \ni u \to \Gamma_{w*}(F)(u)=\supess_{x \in \Omega}f^{lslc}(\nabla u(x))$ can be obtained as the $\Gamma$-limit with respect to the uniform convergence of the sequence of the  integral functionals $(F_p(u))_{p\geq 1}$ defined by  \eqref{curl2}. 
More precisely, in Remark \ref{Lpd1} we will discuss several special cases of assumptions on $f$.

If the supremand  $f(x,\cdot)$ is upper semicontinuous for a.e. $x \in \Omega$, then $f_{\infty}(x,\cdot)=Q_{\infty}f(x,\cdot)$. The same conclusion holds when $f\equiv f(\xi)$.

\noindent In addiction, if $f(x,\cdot)$ is upper semicontinuous and level convex for a.e. $x\in \Omega$, then \eqref{Finfty2} can be specialized, since $$f_{\infty}(x,\cdot)=Q_{\infty}f(x,\cdot)=f^{ls}(x,\cdot) \;\;\hbox{ for a.e. }x \in \Omega.$$ 
The same conclusion holds when $f\equiv f(\xi)$ is level convex.

\noindent Moreover if $N=1$ or $d=1$, 
if $f(x,\cdot)$ is upper semicontinuous or $f\equiv f(\xi)$ then 
we get that $$f_{\infty}(x,\cdot)= Q_\infty f(x,\cdot)= f^{lslc}(x,\cdot)\;\;\hbox{ for a.e. }x \in \Omega.$$

\noindent Note that these results are new in literature since the other $L^p$-approximation results  
suppose that  $f$ is lower semicontinuous with respect to the gradient variable. 
Indeed Theorem 3.2  in \cite{PNodea} requires that $f=f(x,\xi)$ is a Carath\'eodory function satisfying a growth condition with respect to the second variable (uniformly with respect to $x$) of the type \eqref{fcrescita2}; anagously Theorem 3.1 in \cite{CDPP} applies when $f=f(x,\xi)$ is lower semicontinuous w.r.t the second variable.
\\

The paper is organized as follows: Section \ref{prelim} is devoted to preliminaries that will be exploited in the sequel and contains some results of borader scope on explicit representation of envelopes of functions and their effective domains, thus generalizing the results in \cite[Section 2]{RZ}, (cf. \cite{CDeA} for their counterparts in the convex setting). Theorem \ref{relax1} is  proven in Section \ref{relaxsec}, together with an integral representation result for the relaxation of unbounded integral functionals (see Corollary \ref{unbddint}). Finally in Section \ref{Lpapproximation} we provide the proof of Theorem \ref{curlcase2}, and discuss particular cases and special representations in Remark \ref{Lpd1}.

The following notation is adopted in the paper. 
\begin{description}
	\item{-} $(X, \tau)$ denote a topological vector space whose generic elements will be denoted by $x$;
	\item  {-}  for every  $Y \subset X$, by $\overline Y^{\,\tau}$ we mean the closure of $Y$ in $X$ with respect to the topology $\tau$. When $X$ is an Euclidean space and $\tau$ is the natural topology, we adopt just the symbol $\overline Y$;
	\item  {-}  for every set $S \subset X$ we denote by ${\rm co}S$ its convex hull, namely the smallest convex set containing $S$, which can be described as the intersection of all the convex sets (affine hyperplanes which contain $S$). It is easily seen that
	$
	\overline{{\rm co}S}^{\, \tau}= {\rm co}(\overline{S}^{\, \tau});
	$
\item  {-}  $\overline\R$ denotes the set $[-\infty;+\infty]$;
\item  {-} for every function $W:X\to \overline \R$, ${\rm dom}W$ denotes its effective domain, i.e.
	\begin{equation*}\label{dom}
	{\rm dom}W:=\{x \in X: W(x)<+\infty\},
	\end{equation*}
\noindent	and for every $\lambda \in \mathbb R$, $L_\lambda(W)$,
	$$L_\lambda(W):=\{x\in X: W(x)\leq \lambda\}$$
 is the level set of $W$ corresponding to $\lambda$;
	\item {-}  for every $N \in \mathbb N$, $\B_N$ and $\L^N$ denote the Borel measure in $\R^N$, and the Lebesgue one, respectively;

		\item {-} $w$*  denotes the weak* topology in $W^{1,\infty}(\Omega;\R^d)$, unless differently stated. 	
\end{description}


\section{Preliminary results}\label{prelim}
The aim of this section is twofold, from one hand we recall existing results which will be useful in the  body of paper, and from the other, we provide some characterizations of level convex functions defined in general topological vector space $(X,\tau)$. In particular some  of these results are new to our knowledge and of indipendent interest
In Subsection \ref{subsectGamma}, we recall the definition and the main properties of $\Gamma$-convergence. These topics,  together with  classical relaxation results for integral functionals  in the Sobolev setting  (see  Subsection \ref{lscrelint})  enable us to deal with the $L^p$- approximation of Section \ref{Lpapproximation}. 
Finally in Subsection \ref{2.4}  we specialize the properties of the level convex and lower semicontinuous envelope  $f^{lslc}$  when  $f: \R^{d\times N}\to \overline \R$.

%
%
\color{black}
\subsection{Relaxation and level convex envelopes}\label{2.1}
In this subsection we provide several  relations among envelopes of functions in $(X,\tau)$ that will be used in the sequel, thus generalizing some of the results contained in \cite[Section 2]{RZ}.

\begin{defn}\label{levconvx} A function  $F:(X,\tau)\to \overline\R$ is {\sl level convex} if 
$$
F(t x_1 + (1-t)x_2) \leq \max\{F(x_1), F(x_2)\} \quad \forall t \in (0,1),\ \forall x_1, x_2 \in X$$ that is, for every $\la\in \R$  the sublevel set  $L_\lambda(F)$ (see \eqref{Ela}) is convex.
\end{defn}

\begin{defn}  Let   $F:(X,\tau)\to \overline\R$ be a function.
 \begin{enumerate}
 \item The {\sl lower semicontinuous envelope} (or {\sl relaxed function}) of $F$ is defined as 
\begin{equation*}\label{envelopetau}
\Gamma_\tau(F) :=\sup\{G\,|\, G:(X,\tau)\to \overline\R \, ,   G \ \tau\hbox{-lsc}\ \hbox{and } G\le F \hbox{ on } X\}. 
\end{equation*}
\item The {\sl level convex envelope}  of $F$ is defined as 
\begin{equation*}\label{envelopelc}
F^{lc} :=\sup\{G\,|\, G:(X,\tau)\to \overline\R \, ,   G \ \hbox{level convex}\ \hbox{and } G\le F \hbox{ on } X\}. 
\end{equation*}

\end{enumerate}
\end{defn}
\noindent Note that $\Gamma_\tau(F)$ (resp. $F^{lc}$) is the greatest $\tau$-lower semicontinuous (shortly $\tau$-l.s.c) (resp. level convex) function which is less than or equal to $F$. 
	 \noindent By \cite[Proposition 3.5(a)]{DM}  we have that 
\begin{equation}
\label{lsflc0}  
\{\xi \in X: \Gamma_\tau(F)(x)\leq \la\}=\bigcap_{\varepsilon >0}\overline{	L_{\la+\varepsilon}(F)}^{\, \tau}.
\end{equation}
Moreover, by definition, it easily follows that  
	\begin{equation}
	\label{inf1}
	\inf_{X}F=\inf_{X}\Gamma_\tau(F) =\inf_{X}F^{lc}=\inf_{X}\Gamma_\tau(F^{lc}).
	\end{equation}

	\noindent Finally,  if $F:(X,\tau)\to \overline\R$,   we consider the envelope $$F^{lslc}:=\sup\{G\,|\, G:(X,\tau)\to \overline\R \, ,   G \ \hbox{ level convex and }\tau\hbox{-}  l.s.c.\ \hbox{and } G\le F \hbox{ on } X\}, $$
	that is the greatest lower semicontinuous and level convex function less than or equal to $F$.  We recall  that there exists a wide literature devoted to the study of a conjugation for level convex functions (see for example  \cite{BL}, \cite{V} and \cite{M-L} among the others).

\begin{prop}
\label{propertiesF}	
Let $F:X\to \overline\R $. Then 
\begin{equation}\label{envelopelslc}
\Gamma_\tau(F^{lc}) =F^{lslc}\leq (\Gamma_\tau(F)) ^{lc}.
\end{equation}
	In particular 
		if $F$ is level convex  then $\Gamma_\tau(F)$ is level convex and 
		\begin{equation}\label{envelopels}
		\Gamma_\tau(F) =F^{lslc}.
		\end{equation}
\end{prop}

\begin{proof} Since $F^{lslc}$ is $\tau$-l.s.c. and level convex, we have that 
\begin{equation}\label{first}
	F^{lslc}\leq \min\{  \Gamma_\tau(F^{lc}),( \Gamma_\tau(F))^{lc}\} \leq F.
\end{equation}	 
	
\noindent In order to conclude the proof of \eqref{envelopelslc}, observe that for every $ \la \geq \inf F$ and for every $\varepsilon>0$ 
 the set $ \{x \in X: F^{lc}(x)\leq \la+\varepsilon\}$ is convex. Then  its $\tau$-closure is still convex. Thanks to  \eqref{lsflc0},  we can deduce that  $\{x \in X:  \Gamma_\tau(F^{lc})(x)\leq \la\}$ is convex for every $\lambda\geq  \inf F= \inf \Gamma_\tau(F^{lc})$. Thus $ \Gamma_\tau(F^{lc})$ is level convex and lower semicontinuous; consequently, exploiting \eqref{first}, we get \eqref{envelopelslc}. 
	In the particular case when $F$ is level convex,   \eqref{envelopelslc} implies  \eqref{envelopels}.
	\end{proof}
\bigskip

		The following corollary of Proposition \ref{propertiesF} holds:
	\begin{cor}\label{weakseq}  Let $X$ be a separable Banach space, and $X'$ its dual. Let  $F:X'\to \overline \R$ be  level convex and let $$\Gamma_{w*}(F) =\sup\{G:X'\to \overline\R: G\hbox{ weak* lower semicontinuous}, G \leq F\},$$
	 where $w*$ denotes the weak* topology in $X'$. Then $\Gamma_{w*}(F)$ is level convex and  for every $y \in X'$
\beq\label{seqw*} 
\Gamma_{w*}(F)(y)=\inf\{\liminf_n F(y_n): y_n \overset{\ast}{\rightharpoonup} y\},
\eeq

\end{cor}
\begin{proof}  It is sufficient to observe that $\Gamma_{w*}(F)=F^{lslc}$ (where the symbol ${ls}$ refers to the topology $w*$ in $X'$) and to apply \cite[Proposition 2.16]{GP}.
\end{proof}

\begin{prop}
	\label{Functionalslslc}
	For every $F:(X,\tau)\to \overline\R$ and for  every continuous strictly increasing function $\Phi:\overline\R\to [a,b]$, it results
	\begin{equation}
	\label{DM1eq}
	\Gamma_\tau(\Phi(F)=\Phi(\Gamma_\tau (F))\end{equation}
	\begin{equation}
	\label{due}
	(\Phi(F))^{lc}= \Phi(F^{lc})
	\end{equation} 
	and 
	\begin{equation*}
	\label{DMeq2}
	(\Phi(F))^{lslc}=\Phi(F^{lslc}).
	\end{equation*}
	\end{prop}
	\begin{proof}
	\eqref{DM1eq} follows  by   \cite[Proposition 6.16]{DM}.
	In order to show \eqref{due}, note that  $\Phi(F^{lc})\leq \Phi(F)$ implies 
	\begin{equation}
\label{DMeq}
\Phi(F^{lc})\leq (\Phi(F))^{lc}.
\end{equation}
since the composition of an increasing function and a level convex one is still level convex.
Moreover 
$$
\Phi^{-1}((\Phi(F))^{lc})\leq \Phi^{-1}(\Phi(F))=F.
$$
Hence
$$
\Phi^{-1}((\Phi(F))^{lc})\leq F^{lc}.
$$
Thus
	$$
	(\Phi(F))^{lc}\leq \Phi(F^{lc}),
	$$
	which, together with \eqref{DMeq}, gives \eqref{due}. Finally, Proposition \ref{propertiesF}, \eqref{due}  and \eqref{DM1eq}   entail
$$
(\Phi(F))^{lslc}=\Gamma_\tau((\Phi(F))^{lc})=\Gamma_\tau(\Phi(F^{lc}))=\Phi(\Gamma_\tau(F^{lc}))=\Phi(F^{lslc}).
$$
\end{proof}
\bigskip
\begin{rem}
	\label{arctan}{\rm 
 By \eqref{DM1eq} it follows that 
$$
\Gamma_\tau (F)=\Phi^{-1}(	\Gamma_\tau(\Phi(F)).
$$
In particular, if $\Omega\subset \R^{d\times N}$ is a bounded open set, $g:\R^{d\times N}\to \overline \R$ is a Borel  function and  $G:W^{1,\infty}(\Omega;\mathbb R^d)\to \overline\R$ is the supremal functional defined as $$G(u):=\supess_{\Omega}g(\nabla u),$$   in order to detect $\Gamma_\tau(G)$, it suffices to detect $\Gamma_\tau(\arctan G)$. Since it  holds  that  $$(\arctan G)(u)=\supess_{\Omega}\arctan g(\nabla u),$$ without loss of generality, we can assume that $g$ is finite valued.}
\end{rem}
\bigskip

\noindent We conclude this subsection by proving a  general representation result for the functional $F^{lc}$. 

\begin{prop}
	\label{flcbuild}
	Let $F:(X,\tau)\to \overline \R.$ Then
	\begin{equation}\label{Flcrep}
	F^{lc}(x)=\inf\{\lambda \in \mathbb R: x \in {\rm co}L_\lambda (F)\}.
	\end{equation}
\end{prop}
\begin{proof}[Proof]
	Let $$\iota:x \in X \to \inf\{\lambda \in \mathbb R: x \in {\rm co}L_\lambda (F)\}.$$
	Clearly 
	\begin{equation}
	\label{ileqf}
	\iota(x)\leq F(x) \hbox{ for every } x \in X.
	\end{equation}
	Moreover $i$ is level convex. Indeed, fixed $x_1,x_2\in X$, for every $\varepsilon >0$ there exists $\lambda_1$ and $\lambda_2$ such that
	\begin{align*}
	\lambda_1 < \iota(x_1)+\varepsilon, \hbox{ and }
	\lambda_2 < \iota(x_2)+\varepsilon,
	\end{align*}
	and 
	\begin{align*}
	x_1 \in {\rm co}L_{\lambda_1}(F), \hbox{ and } x_2 \in {\rm co}L_{\lambda_2}(F).
	\end{align*}
	Thus
	$$tx_1 + (1-t)x_2 \in  {\rm co}L_{{\rm max}\{\lambda_1, \lambda_2\}}(F), \; \; t \in [0,1],
	$$	
	and so 
	$$
	\iota(t x_1+ (1-t)x_2) \leq \max\{\lambda_1,\lambda_2\} < \max\{\iota(x_1), \iota(x_2)\} + 2 \varepsilon.
	$$
	The arbritrariness of $\varepsilon$ guarantees the level convexity of $\iota$, which together with \eqref{ileqf}, guarantees
	\begin{equation}
	\label{ileqflc}
	\iota(x)\leq F^{lc}(x) \hbox{ for every } x \in X.
	\end{equation}
	In order to prove the opposite inequality we have that 
	for every level convex $G \leq F$,
	\begin{align*}
	L_{\lambda}(G)\supseteq L_\lambda (F) \hbox{ for every }\lambda,
	\end{align*}
	which gives
	\begin{align*}
	{\rm co}L_\lambda(F)\subseteq {\rm co}L_\lambda(G)=L_\lambda(G) \hbox{ for every }\lambda \in \mathbb R. 
	\end{align*}
	Since for every $x \in X$ it results 
	\begin{align*}
	G(x)=\inf\{\lambda \in \mathbb R: x \in L_\lambda (G)\}\leq
	\inf\{\lambda \in \mathbb R: x \in {\rm co }L_\lambda(F)\}= \iota(x),
	\end{align*}
	by choosing $G= F^{lc}$ we get 
	$ \iota \geq F^{lc}.$
	The latter inequality and  \eqref{ileqflc} conclude the proof.

\end{proof} 
%
\subsection{$\Gamma$-convergence}\label{subsectGamma}
Now we recall the notion of $\Gamma$-convergence for family of functionals defined in the topological space $(X,\tau)$, (for more details on the theory we refer to \cite{DM}).
To this end, we denote by $\mathcal{U}(x)$ the set of all open neighbourhoods of $x$ in $X$ .
\begin{defn} Let $F_n:X \to \R$ be a sequence of functions. 
	The {\sl $\Gamma({\tau})$-lower limit} and the {\sl $\Gamma({\tau})$-upper limit} of the sequence
	$(F_n)$  are the functions from $X$ into $ \overline{\R}$  defined by
	$$ \Gamma({\tau})\hbox{-}\liminf_{n\to \infty} F_n(x) :=\sup_{U\in \mathcal{U}(x) \  }\liminf_{n\to \infty} \inf_{y\in U}F_n(y) $$
	$$ \Gamma({\tau})\hbox{-}\limsup_{n\to \infty} F_n(x) :=\sup_{U\in \mathcal{U}(x) \  }\limsup_{n\to \infty} \inf_{y\in U}F_n(y) $$
	If there exists a function $F:X\to \overline{\R}$  such that $\Gamma (\tau) \hbox{-}\liminf\limits_{n\to \infty} F_n= \Gamma({\tau})\hbox{-}\limsup\limits_{n\to \infty} F_n,$
	then we write $$F=\Gamma({\tau})\hbox{-}\lim_{n\to \infty} F_n$$ and we say that the sequence  $(F_n)$
	$\Gamma(\tau)$-converges to $F$ or that $F$ is the $\Gamma({\tau})$-limit of $(F_n)_{n}$.
	
\end{defn}
Given a family of functionals $G_\ee :X \ds \overline{\R}$, we say that $(G_\ee)_{\ee}$ 
$\Gamma(\tau)$-converges to the functional $G$, as $\ee\to 0$,  
if for every $(\ee_n)\to 0$ the sequence $(G_{\ee_n})$ $\Gamma({\tau})$-converges to $G.$

The introduction of this variational convergence by De Giorgi and Franzoni (see \cite{DM} and the bibliography therein) is motivated by the  next theorem.  Indeed, under the assumption of equicoercivity for the sequence $(F_n)$, it holds the
important  property of convergence of the minimum values.

\begin{thm}\label{convmin} Suppose that the sequence $(F_n)$ is equi-coercive  in $X$, i.e. for every $t\in \R$ there exists a closed  compact subset $K_t$ of $X$ such that $\{ F_n \le t \} \subset K_t$ for every $n\in\N$.
	\noindent If $(F_n)$ $\Gamma({\tau})$-converges to a function $F$ in X,  then
	\beq\nonumber
	\min_{x\in X} F(x)=\lim_{n\to \infty} \inf_{x\in X} F_n(x).
	\eeq
	Moreover if $x_n$ is such that $F_n(x_n)\leq \inf_X F_n+\ee_n$, where $\ee_n\to 0$ and $x_{n_k}\to x$ for some
	subsequence $(x_{n_k})_k$ of $(x_n)$ then $F(x)=\min_X F$.
	
\end{thm}
\noindent For a proof, see \cite[Theorem 7.8 and Corollary 7.17]{DM}.\\

In the following proposition we summarize some properties of the $\Gamma$-convergence useful in the sequel (see \cite[Proposition 6.8, Proposition 6.11, Proposition 5.7, Remark 5.5,  Proposition 6.26]{DM}).

\begin{prop}\label{gammaprop}  Let $F_n: X \ds \overline{\R}$ be a sequence of functions.  
	Then 
	\begin{itemize}
		\item[\rm 1.]  
Let $\hat F:= \Gamma(\tau)\hbox{-}\lim\limits_{n\to \infty}F_n$, then $\hat F$ is $\tau$-lower semicontinuous on $X$;
		\item [\rm 2.] if $(F_n)$ is a not increasing sequence which  pointwise converges to $F$  then  $\Gamma(\tau)\hbox{-}\lim\limits_{n\to \infty} F_n= \Gamma_\tau (F)$. In particular if $F_n=F$ for every $n\in \N$ then  $\Gamma(\tau)\hbox{-}\lim\limits_{n\to \infty} F= \Gamma_\tau (F);$
\item[\rm 3.]  if $\Gamma_{\tau}(F_n)$ is the $\tau$-l.s.c. envelope of  $F_n$, 
 then the sequence  $(F_n)$ $\Gamma(\tau)$-converges to $F$ if and only if the sequence of the relaxed functions $(\Gamma_{\tau}(F_n))$
 $\Gamma(\tau)$-converges to $F$, and   
$$\Gamma(\tau)\hbox{-}\lim\limits_{n\to \infty} F_n=\Gamma(\tau)\hbox{-}\lim\limits_{n\to \infty}\Gamma_{\tau}(F_n);  $$
%

		\item [\rm 4.] if $(F_n)$ is an increasing sequence of $\tau$-lower semicontinuous functions which  pointwise converges to $F$  then  $\Gamma(\tau)\hbox{-}\lim\limits_{n\to \infty} F_n= F$.
		\item [\rm 5.] for every $c\in \R$   $\Gamma_{\tau}(\max\{F,c\})=\max \{\Gamma_\tau (F),c\}.$
		
	\end{itemize}
\end{prop}

Next we  recall  the sequential
characterization of $\Gamma({\tau})$-liminf, $\Gamma({\tau})$-limsup and
$\Gamma({\tau})$-limit when the topological space $(X,\tau)$ satisfies  the first
axiom of countability (for a proof see \cite[Proposition 8.1]{DM}). 

\begin{prop} \label{seqcharacsupinf}
	Let $F_n:X \ds \overline \R $ be a sequence of functions. Then the  function

$${\dis F(x)=\Gamma({\tau})\hbox{-}\lim\ninf F_n(x)}$$ is characterized by the
	following inequalities:
	\begin{itemize}
		\item[-]($\Gamma$-liminf inequality) for every $x \in X$ and for every sequence $(x_n)$
		converging to $x$ in $X$ it is
		$$ F(x)\leq \liminf\ninf F_n(x_n); $$
		\item[-] ($\Gamma$-limsup inequality) for every $x \in X$ there exists a sequence $(x_n)$ (called a {\rm recovering sequence}) converging to $x$ in $X$   such that
		$$ F(x)= \lim\ninf F_n(x_n).  $$
	\end{itemize}
\end{prop}

Finally we note that the level convexity is stable under both  pointwise and $\Gamma$-convergence (for a proof see  \cite[Proposition 2.11]{GP}). 
\begin{prop}\label{lcGammalimite} Let $(X,\tau)$ be a topological vector space and let $F_n:X\to \overline{\R}$ be  a sequence of level convex functions.  Then 
	\begin{itemize}
		\item [\rm 1.] the function $ F^{\#}(x)=\limsup \limits_{n\to\infty} F_n(x)$   is level convex;
		\item  [\rm 2.] the function $\Gamma(\tau)$-$\limsup\limits_{n\to\infty} F_n$ is level convex. 
	\end{itemize}
	
\end{prop}

\subsection{Lower semicontinuity and relaxation result in the integral setting}
\label{lscrelint}

In the sequel we collect some definitions and results that will be crucial for the proof of Theorem \ref{curlcase2}. We refer the reader to \cite{B} and \cite{Dac} for a detailed treatment of this subject.

\begin{defn}\label{defqcx} Let $g:\R^{d\times N}\to\R$ be a  Borel  function and let $Q:=]0,1[^N$. Then $g$ is said quasiconvex (in the sense of Morrey) if $$g(\xi)\leq \int_Q g( \xi+ \nabla u(y))\dy$$
	for every  $u\in W^{1,\infty}_0(Q; \R^d)$ and $\xi\in \R^{d\times N}$.
\end{defn}

By  \cite[Theorem 5.3(4)]{Dac} it follows that any quasiconvex function is continuous.
The quasiconvexity is a sufficient (and necessary) condition for the lower semicontinuity of an integral functional on $W^{1,p}(\Omega; \R^d)$  with respect to the weak topology of $W^{1,p}(\Omega; \R^d)$. In order to state such results,  let $1 \leq p < +\infty$ and let $g:\mathbb R^{d\times N}\to \R$ be a quasiconvex function, such that \beq\label{fcrescitapb} 
0\leq g(\xi) \leq \beta (1+|\xi|^p) \qquad \hbox{ for every }
\xi\in \R^{d\times N}.
\eeq
Let $G:W^{1,p}(\Omega;\mathbb R^d)$ be the integral functional defined by
\begin{equation*}
\label{intG}
G(u):=\int_\Omega g(\nabla u(y))dy.
\end{equation*}
Then  $G$ is sequentially weakly lower semicontinuous on $W^{1,p}(\Omega; \R^d)$ (see \cite[Theorem 8.4]{Dac}).

%
%
%

If the Borel function $g:\R^{d\times N}\to\R$ fails to be quasiconvex, one can introduce its quasiconvex envelope, namely 
\begin{align}\label{Qg}
Qg:=\sup\{h\,|\, h:\R^{d\times N}\to\R\, ,   h \ \hbox{quasiconvex}\ \hbox{and } h\le g\}. 
\end{align}
	
	\begin{rem}\label{Qgcont}{\rm
It is worth to observe that, being $Qg$ quasiconvex, then $Qg$ is a continuous function (see \cite[Lemma 5.42]{AFP} and \cite[Theorem 5.3]{Dac}).}
\end{rem}

\noindent The following representation formula holds:

\begin{thm}\label{repqcx} {\rm \cite[Theorem 6.9]{Dac}} Let $g:\R^{d\times N}\to\R$ be a  Borel and locally bounded function. Assume that there exists a quasiconvex function $h:\R^{d\times N}\to\R$ such that $g\geq h$. Then for every $\xi\in \R^{d\times N}$ 
	$$Qg(\xi)=\inf\Bigl\{ \int_Q g( \xi+ \nabla u(y))\dy \ :\ u\in W^{1,\infty}_0(Q; \R^d) \Bigr\}\,.$$
\end{thm}

%
%
%
  The following result which holds under very general assumptions, i.e.  when $g=g(x,\xi)$ is only $\L^N \otimes \B_{d\times N}$-measurable function, will be crucial in the proof of Theorem \ref{curlcase2}.
\begin{thm}\label{relaxint2}{\rm \cite[Theorem 4.4.1]{B}} Let $\Omega\subseteq \R^N$ be a bounded open set, let  $1\leq p<+\infty$ and let $g:\Omega\times \R^{d\times N}\to[0,+\infty)$ be a $\L^N\otimes \B_{d\times N}$-measurable function, satisfying \eqref{fcrescitapb}.   
	Then there exists a Caratheodory function $\tilde g:\Omega\times \R^{d\times N}\to[0,+\infty)$,  quasiconvex  in the second variable,  such that  $$\Gamma_{w_{seq}}(G)(u)=\int_\Omega \tilde g(x, \nabla u(x))dx \hbox{ for every } u\in W^{1,p}(\Omega; \R^d),$$
	where $\Gamma_{w_{seq}}(G)$  denotes the sequential lower semicontinuous envelope of $G$ with respect to the weak topology in $W^{1,p}(\Omega;\R^d)$.
	 
	\noindent Moroever $$Qg(x,\xi)\leq \tilde g(x,\xi)  \hbox{ for a.e. } x\in \Omega, \hbox{ and for every } \xi\in \R^{ d\times N}.$$
	
	
\end{thm}

\begin{rem}\label{remBut} {\rm By  \cite[Remark 4.4.5]{B} it follows that 
\begin{itemize} 
\item[1.]  if $g=g(\xi)$  then $Qg= \tilde g$;
\item[2.] if  $g(x,\cdot)$ is upper semicontinuous for a.e. $x\in \Omega$  then 
$Qg(x,\xi)= \tilde g(x,\xi)$   for a.e.  $x\in \Omega$, and for every $\xi\in \R^{ d\times N}.$
\end{itemize}}
\end{rem}

\subsection{Envelopes of real functions}\label{2.4}

In this subsection we detail the results of subsection \ref{2.1} in the special case when $X=\mathbb R^{ d\times N}$ and $\tau$ is the natural topology.
\begin{defn}
	\label{lslcdef}
	Let $f:\mathbb R^{ d\times N}\to \overline\R$ be a function. Set $$\mathcal F_{lc}(f):=\{ g:\mathbb R^{ d\times N}\to \overline \R:\  g\leq f,\ g  \hbox{ level convex} \},$$ 
  $$\mathcal F_{ls}(f):=\{ g:\R^{ d\times N}\to {\overline\R}:\  g\leq f,\ g \hbox{ $\tau$-lower semicontinuous} \},$$ 
 and
  $$\mathcal F_{lslc}(f):=\{ g:\R^{ d\times N}\to {\overline\R}:\  g\leq f,\ g  \hbox{ $\tau$-lower semicontinuous and level convex} \}.$$
 
\noindent Consequently define $f^{lc}, f^{ls}, f^{lslc}:\R^{d\times N}\to \overline{\R}$, as
	 $$f^{lc}(\xi):= \sup\{ g(\xi):\  g\in \mathcal F_{lc}(f)  \},$$
	$$f^{ls}(\xi):= \sup\{ g(\xi):\  g\in \mathcal F_{ls}(f)  \},$$
and $$f^{lslc}(\xi):= \sup\{ g(\xi):\  g\in \mathcal F_{lslc} (f) \}.$$
\end{defn}
\noindent An explicit formula to compute $f^{lc}$ in given by Proposition \ref{flcbuild}, applied to $F=f$ and to $X=\mathbb R^{d \times N}$.
 	\begin{rem}\label{fslcls1} If $f:\R^{ d\times N}\to \overline\R$ , 
	by \eqref{envelopelslc}  we have that 
\begin{equation}\label{eqineq}
	f^{lslc}=(f^{lc})^{ls}\leq (f^{ls})^{lc}.
	\end{equation}
Then   if  $f$ is level convex, we get  that
	\begin{equation}
	\label{fslcls}
	f^{lslc}= f^{ls}.
	\end{equation}  
	In particular, thanks to \cite[Theorems 3.4 and 2.7]{BJW},  we get that $f^{ls}$ is a  strong Morrey quasiconvex  function less than or equal to $f$.
\end{rem}

\begin{rem}
	\label{noteqineq}{\rm 
	In general $(f^{lc})^{ls}\lneq (f^{ls})^{lc}$, since the level convex envelope of a lower semicontinuous function might not be lower semicontinuous. To this end, it suffices to consider the function
	$$\chi_{\mathbb R^2\setminus C}=\left\{\begin{array}{ll}
	 0 &\hbox{ if }x \in C,\\
	 1 &\hbox{otherwise}
	 \end{array}
	 \right.$$
	where $C:=\{(x_1,0): x_1 \in \mathbb R\}\cup \{(0,1)\}$. Indeed $\chi_{\mathbb R^2\setminus C}$ is lower semicontinuous but not level convex. On the other hand 
	$(\chi_{\mathbb R^2\setminus C})^{lc}=\chi_{\mathbb R^2\setminus D}$, where $D=\{(x_1,x_2): x_1 \in \mathbb R, 0\leq x_2< 1\}\cup \{(0,1)\}$, which is not closed. Clearly $(\chi_{\mathbb R^2\setminus C})^{lslc}= \chi_{\mathbb R^2\setminus \overline D}< \chi_{\mathbb R^2\setminus D}$.}
\end{rem}

Now we are in position to show a result characterizing the effective domain of $f^{lc}$, based on Carath\'eodory's Theorem (see \cite[Theorem 1.2.1]{CDeA}).

\begin{prop}
	\label{domflc}
	For every $f:\R^{d\times N} \to \overline \R$, it results
	\begin{equation*}\label{1.2.6CDA}
	{\rm dom}(f^{lc})={\rm co}({\rm dom}f)
	\end{equation*}
\end{prop}
\begin{proof}
	The result is achieved by proving a double inequality. If $\xi \in {\rm dom}(f^{lc})$, then there exists $\lambda \in \mathbb R$ such that $\xi \in {\rm co}{L_\lambda}(f)$, thus there exist $\xi_1, \dots, \xi_{d\times N+1} \in L_\lambda(f)$ and $t_i \in [0,1], i=1,\dots d\times N+1$ such  that $\sum_{i=1}^{d\times N+1}  t_i =1 $ and $\xi= \sum_{i=1}^{d\times N+1}t_i \xi_i$. Clearly $\xi_1,\dots \xi_{d\times N+1} \in {\rm dom }f$, hence $\xi \in {\rm co}({\rm dom}f)$. Thus  it remains to prove the opposite inequality: 
	if $\xi \in{\rm co}({\rm dom}f)$, again thanks to Carath\'eodory's Theorem there exist almost $d\times N+1$ points $\xi_1, \xi_2,\cdots, \xi_{d\times N+1} \in {\rm dom}f$ and $t_1,\cdots,t_{d\times N+1} \in [0,1]$ such that  $\sum_{i=1}^{d\times N+1}  t_i =1 $ and  $\xi= \sum_{i=1}^{n+1}  t_i \xi_i $. Hence there exists $\lambda \in \mathbb R$ such that $f(\xi_i) \leq \lambda$ for every $i\in \{ 1,\cdots,n+1\}.$
	Consequently Proposition \ref{flcbuild} entails that $f^{lc}(\xi)\leq \lambda$, i.e. $\xi \in {\rm dom}(f^{lc})$ and this concludes the proof.
\end{proof}

\begin{prop}
	\label{Lcqith coerci}
Let $f:\mathbb R^{ d\times N}\to \overline \R$, then,  for every $\la\in \R$
\begin{itemize}
\item[1.] ${\rm co}L_\lambda(f)\subseteq L_\lambda(f^{lc})$;  
\item [2.] if $f$ is coercive $($i.e. $\lim_{|\xi|\to \infty} f(\xi)=+\infty$$)$,  then  \begin{equation}\label{casocoer1} L_\lambda(f^{lc})\subseteq \overline{L_\lambda(f^{lc})}\subseteq{\rm co}L_\lambda(f^{ls}). \end{equation}In particular 
 \begin{equation}\label{casocoer} L_\lambda(f^{lc})\subseteq co(\bigcap_{\varepsilon >0}\overline{L_{\la+\ee}(f)});
 \end{equation}
\item [3.]   if  $f$ is lower semicontinuous and coercive then ${\rm co}L_\lambda(f)= L_\lambda(f^{lc})$.

\end{itemize}
\end{prop}
\begin{proof}[Proof]
\begin{itemize}
\item[1.] It follows by the convexity of $L_\lambda(f^{lc})$ and by the fact that $L_\lambda (f)\subseteq L_\lambda (f^{lc})$.
 \item [2.]
 Assume that $f^{lc}(\xi)\leq \lambda$. By Proposition \ref{flcbuild} there exists a sequence $(\lambda_n)$ converging to $f^{lc}(\xi)$ such that $\xi\in co(L_{\lambda_n} (f)).$
In particular, thanks to the Carath\'eodory's Theorem,  for every $n\in \N$ there exist
	$\xi_n^1, \xi_n^2,\cdots, \xi^{d\times N+1}_n\in L_{\lambda_n}(f)$ and $t^i_n\in[0,1]$, $i\in \{1,\cdots,d\times N+1\}$ such that $\xi=\sum_{i=1}^{d\times N+1} t^i_n \xi_n^i $ and $\sum_{i=1}^{d\times N+1} t^i_n=1 $. 
	Since $ L_{\lambda_n}(f)$  is  bounded by coercivity, without loss of generality, we can assume, up to the extraction of not relabelled subsequences, that for every $i\in \{1,\cdots,d\times N+1\}$ there exist 	 $\lim_{n\to\infty}\xi_n^i=\xi^i$ and  $\lim_{n\to\infty} t^i_n =\bar t^i.$ In follows that $\xi=\sum_{i=1}^{d\times N+1} \bar t^i \xi^i $ and $\sum_{i=1}^{d\times N+1} \bar t^i=1 $.
	By definition of $f^{ls}$ it follows that 
	$$f^{ls}(\xi^i)\leq \liminf_{n\to \infty} f(\xi^i_n)\leq \lim_{n\to \infty} \lambda_n=f^{lc}(\xi)\leq \lambda.$$
	 Therefore $\xi\in 	  {\rm co} L_\lambda(f^{ls})$ and \eqref{casocoer1} follows.
 By  \eqref{casocoer1} and (\ref{lsflc0}), we obtain \eqref{casocoer}.
	\item [3.]  It follows by 1. and 2.
	
			\end{itemize}
	\end{proof}

\begin{rem}{\rm Let $f:\mathbb R^{d\times N} \to \mathbb R$ be defined by
	$$f(\xi)=\left\{\begin{array}{ll}
	|\xi| &\hbox{ if }\xi \not= 0,\\
	1 &\hbox{ if }\xi=0.\end{array}\right.$$
	Then 
	$L_0(f)=\emptyset$, so ${\rm co}L_0(f)=\emptyset$, while $f^{ls}(\xi)=f^{lc}(\xi)= |\xi|$, and so ${L_0}(f^{lc})=\{0\}$.
Thus we cannot expect equality in (i). Moreover this  example proves  also  that in general $ L_\lambda(f^{lc})\not\subseteq \overline{{\rm co}L_\lambda(f)}$ and $ L_\lambda(f^{ls})\not= \overline{L_\lambda(f)}$ showing the sharpness of (\ref{lsflc0}).
}
	\end{rem}
	
The following result specializes \eqref{lsflc0} when $X= \R^{d\times N}$, thus providing a useful description of  the sublevel sets of $f^{lslc}$.
	
\begin{prop}
\label{properties}	
Let $f:\mathbb R^{d\times N}\to \overline\R$. 
Then for every $\la\in \R$ it holds
	  $$L_{\la}(f^{lslc})=\bigcap_{\varepsilon >0}\overline{ co(L_{\la+\ee}(f))}.$$
\end{prop}

\proof	
First of all, we notice that, thanks to (\ref{eqineq}) and   (\ref{lsflc0}), we have that 

\begin{equation}\label{1}
	L_\lambda(f^{lslc}) = L_\lambda((f^{lc})^{ls}) =\bigcap_{\delta >0} \overline{L_{\la+\delta}(f^{lc})	},
	\end{equation}
	in particular, Proposition \ref{Lcqith coerci}(1) implies 
	\begin{equation*}\label{1b}
	\bigcap_{\varepsilon >0}\overline{ co(L_{\la+\ee}(f))} \subseteq L_{\la}(f^{lslc}).
	\end{equation*}
	The proof of the opposite inclusion will be developed in several steps.

\begin{itemize}
\item[Step 1. ] First we consider the case when  $f$ is coercive. 
\noindent Under this extra assumption, by applying \eqref{casocoer1}, we have that for every $\la\in \R$ and for every $\delta>0$
\begin{equation}\label{2}\overline{L_{\lambda+\delta}(f^{lc})}\subseteq co(\bigcap_{\varepsilon >0} \overline{L_{\la+\delta+\ee}(f)}).
\end{equation}
By putting together \eqref{1} and \eqref{2}, it follows
\begin{eqnarray}
\nonumber
L_\lambda(f^{lslc}) & \subseteq & \bigcap_{\delta >0}co(\bigcap_{\varepsilon >0} \overline{L_{\la+\delta+\ee}(f)}) \subseteq \bigcap_{\delta >0} \bigcap_{\varepsilon >0} co(\overline{L_{\la+\delta+\ee}(f)})\\
	\nonumber
&=&  \bigcap_{r >0} co(\overline{L_{\la+r}(f)})= \bigcap_{r >0}\overline{ co(L_{\la+r}(f))}
\nonumber
\end{eqnarray}
and this identity concludes the proof  in the coercive case.

%
%
\item[Step 2. ]  In the second step we consider the general case  when $f:\R^{d\times N}\to [0,+\infty]$. We define $f_n(\xi):=\max\{f(\xi), \frac 1 n|\xi|\}.$ 
Since $f\leq f_n$ then $f^{lc}\leq (f_n)^{lc}:=f_n^{lc}$ 
that implies 
   $$f^{lc}\leq (f^{lc})_n\leq f_n^{lc}\leq  f_n$$ for every $n\in \N$.
In particular $$f^{lc}\leq \inf_n f_n^{lc}\leq  \inf f_n=f.$$ 

By Proposition \ref{lcGammalimite}(1), since $(f_n^{lc})$ is monotone, the function $g(\xi):=\inf_n f_n^{lc}(\xi)$ is  level convex. Then $$f^{lc}=\inf_n f_n^{lc}.$$
\noindent Since $f_n$ is coercive,  by applying \eqref{casocoer1}, we have  that for every $n\in \N$, for every $\la \geq 0$ and $\ee>0$ $$L_{\lambda+\ee}(f_n^{lc})\subseteq  co L_{\lambda+\ee}(f_n^{ls})\subseteq  co L_{\lambda+\ee}(f^{ls}).$$
Now,  for fixed  $\la \geq 0$ and $\ee>0$,  if $\xi\in L_\lambda(f^{lc})$ then  for $n=n(\xi)$ big enough we get that  $\xi\in L_{\lambda+\ee}(f_n^{lc}).$
Thus 
 $$L_\lambda(f^{lc})\subseteq   co L_{\lambda+\ee}(f^{ls})$$
that implies  
\begin{equation}\label{3}
\overline{L_\lambda(f^{lc})}\subseteq   \bigcap_{\ee>0}co L_{\lambda+\ee}(f^{ls}).
\end{equation}
Thanks to \eqref{1}, \eqref{3}  and Proposition \ref{Lcqith coerci} (1), we have that 
\begin{eqnarray}
\nonumber
L_\lambda(f^{lslc}) &=&\bigcap_{\delta >0}\overline{	L_{\la+\delta}(f^{lc})}\subseteq  \bigcap_{\delta >0}\bigcap_{\ee>0}co L_{\lambda+\delta+\ee}(f^{ls})
	\\
	\nonumber
	&=& \bigcap_{\ee>0}co L_{\lambda+\ee}(f^{ls}) =  \bigcap_{\ee>0}co  \big(\bigcap_{\delta >0}\overline{L_{ \lambda+\varepsilon+\delta}	( f )}\big)\subseteq 
	\bigcap_{\ee>0} \bigcap_{\delta >0} co \overline{L_{ \lambda+\varepsilon+\delta}	( f )}\\
	\nonumber
&=& \bigcap_{r >0}  co \overline{L_{ \lambda+r}	( f )}= \bigcap_{r >0}\overline{ coL_{\la+r}(f)}
\end{eqnarray}
and this identity concludes the proof.
\item[Step 3. ]  Now we consider the  case  when $f:\R^{d\times N}\to \bar \R$ is such that $\inf f>-\infty$.
Then, it is sufficient to apply the previous step to the non negative function $g:=f- \inf f$ and use the fact that $g^{lc}=f^{lc}-\inf f$ and  $g^{lslc}=f^{lslc}-\inf f$.
\item[Step 4. ] 
Finally, when $f:\R^{d\times N}\to \bar \R$ is such that  $\inf f=-\infty$ we can consider the approximation $\f_n:=\max\{f,-n\}\geq f$.
Then for every $n\in \N$ and $\la \geq -n$, thanks to the previous step, it holds
	  \beq\label{fnlc}\bigcap_{\varepsilon >0}\overline{ co(L_{\la+\ee}(\f_n))}=L_{\la}(\f_n^{lslc}).\eeq

Denote by $\f_n^{lc}$ the function $(\f_n)^{lc}$. Then  $f^{lc}\leq \f_n^{lc}\leq  \f_n$ for every $n\in \N$.
In particular  $$f^{lc}\leq \inf_n \f_n^{lc}\leq  \inf \f_n=f.$$ Applying again Proposition \ref{lcGammalimite}(1),  in light of the monotonicity of $(\varphi_n^{lc})$, it turns out that $g:=\inf_n \f_n^{lc}$ is  level convex. Then $$f^{lc}=\inf_n \f_n^{lc}$$ and, by Proposition \ref{gammaprop} (2)-(3) we have that 
$$f^{lslc}=\Gamma\hbox{-}\lim_{n\to \infty} \f_n^{lc}=\Gamma\hbox{-}\lim_{n\to \infty} \f_n^{lslc}.$$
Then for fixed  $\la \in \R$, and $\xi\in  L_\lambda(f^{lslc})$ there exists a sequence $(\xi_n)$ converging to  $\xi$ such that  for every $\ee>0$ one can find \color{black} $n_0=n_0(\ee)$ such that $$\f_n^{lslc}(\xi_n)\leq  f^{lslc} (\xi)+\ee\leq \lambda +\ee \quad \forall  n\geq n_0$$
that is 
$$(\xi_n)_{n\geq n_0}\subseteq \bigcup_{n\geq n_0} L_{\lambda+\ee}(\f_n^{lslc}) $$  that implies 
$$\xi \in \overline{ \bigcup_{n\geq n_0} L_{\lambda+\ee}(\f_n^{lslc}) }.$$
 By applying \eqref{fnlc}, we get that, for every $\ee>0$ there exists $n_0=n_0(\ee)\in \N$ such that

$$\xi \in  \overline{\bigcup_{n\geq n_0}  \bigcap_{\delta >0}\overline{ co(L_{\la+\ee+\delta}(\f_n))}}  .$$

Since $f\leq \f_n$ for every $n\in \N$ it follows that for every $\ee>0$
$$ \xi  \in \overline{\bigcup_{n\geq n_0}  \bigcap_{\delta >0}\overline{ co(L_{\la+\ee+\delta}(f))}} = \bigcap_{\delta >0}\overline{ co(L_{\la+\ee+\delta}(f))},$$
that implies
$$ L_\lambda(f^{lslc})\subseteq \bigcap_{\ee>0} \bigcap_{\delta >0}\overline{ co(L_{\la+\ee+\delta}(f))}= \bigcap_{\ee>0} \overline{ co(L_{\la+\ee}(f))}. $$

$$ L_\lambda(f^{lslc})\subseteq \overline{\bigcup_{n\geq n_0}  \bigcap_{\delta >0}\overline{ co(L_{\la+\ee+\delta}(\f_n))}}  .$$


\end{itemize}
\qed

\section{Relaxation results}\label{relaxsec}
This section is mainly devoted to the proof of Theorem \ref{relax1}.  First of all we give an equivalent formulation of assumption (H).

 \begin{rem}\label{equiv} {\rm Assumption (H) is equivalent to require  the following property:\\

$(H')$: $f$ is level convex and  there exist two sequences $(\xi_n) \subseteq \R^{d\times N}$ and $(\la_n) \searrow  \inf_{\mathbb R^{d \times N}}f $  such that

\begin{equation}
 	\label{H1}
 	f(\xi_n)\leq \la_n  \hbox{ and } \limsup_{\xi \to 0} f(\xi_n+ \xi)\leq  \la_n\quad  \forall n\in \N.
 	\end{equation}

%
 Indeed, sssume that $(H)$ holds. In order to show that $f$ is level convex, it remains to  check that when $\inf_{\mathbb R^{d \times N}}f= \min_{\R^{d \times N}}f=:\bar \la \in \R$ the sublevel set $L_{\bar \la}(f)$ is convex. This holds  since the sublevel set corrisponding to the minimum value $\bar \la$ satisfies  $$L_{\bar \lambda}(f)=\bigcap_{\lambda >\bar{\lambda}}L_\lambda(f)$$ and $L_\lambda (f)$ is  convex for every $ \lambda >\bar{\lambda}$ by hypothesis. In order to prove \eqref{H1} it suffices to take $(\la_n)$ such that $(\la_n) \to \inf_{\mathbb R^{d \times N}}f$ and choose $\xi_n$  in the interior of 
 $E_{\la_n}.$
 

 Viceversa, assume that $(H')$ holds, thus $L_{\la}(f)$ is convex for any $\la\in \R$ such that  $\la \geq \inf_{\mathbb R^{d \times N}}f$. 
In order to show  that $L_\la(f)$ has nonempty interior for any $\la> \inf_{\mathbb R^{d \times N} } f$, let us  choose   $n$ big enough such that $\la_n<\la$. Let  $0<\epsilon<\la-\la_n$.Thanks to  \eqref{H1} the set    $L_{\la_n+\epsilon}(f)$  has nonempty interior and since $L_{\la_n+\epsilon}(f)\subseteq  L_{\la}(f)$, the same holds for $L_{\la}(f)$. }
 \end{rem}

The proof of Theorem \ref{relax1} relies on the following result, which is a consequence of \cite[Theorem 2.1]{AHM} and exploits arguments as in \cite[Theorem 3.1]{GP}.

\begin{thm}\label{AHM}
	Let $I_C: \mathbb R^{d\times N} \to [0,+\infty]$ be the indicator function of a nonempty open bounded convex set $C$, such that $\underline 0 \in C$, i.e.
	\begin{equation}\label{indicatorC}I_C(\xi):=\left\{\begin{array}{ll}
	0 &\hbox{ if }\xi \in C,\\
	+\infty &\hbox{ if }\xi \not \in C. 
	\end{array}\right.
	\end{equation}
	Let ${\mathcal I}, \overline{\mathcal I}:W^{1,\infty}(\Omega; \mathbb R^d)\to[0,+\infty]$ be the functionals defined  by
	\begin{equation}
	\label{Ical}
		{\mathcal I}(u):= \int_{\Omega}I_C(\nabla u)dx,
		\end{equation}
 and 
\begin{equation}\label{Ibarcal}\overline{\mathcal I}( u):=\int_{\Omega}I_{\overline C}(\nabla u)dx.
	\end{equation}
	Then  $$\Gamma_{L^1}(\mathcal I )(u)=\overline{\mathcal I}(u) \quad
	\hbox{ for every }u \in W^{1,\infty}(\Omega;\mathbb R^d).$$
	

\end{thm}

\begin{rem}
	\label{remAHM} 
		{\rm Note that when $\Omega$ is a   bounded open subset  with Lipschitz boundary
		$$C \hbox{ bounded} \Longrightarrow  \Gamma_{w^*}(\mathcal I)= \Gamma_{w^*_{seq}}(\mathcal I)=\Gamma_{L^\infty}(\mathcal I)=\Gamma_{L^1}(\mathcal I).$$
	Since $\mathcal I(u)$ is finite if and only if  $\nabla u(x) \in C$ for a.e. $x\in \Omega$,   the first equality is a consequence of Banach-Alaoglu-Bourbaki's Theorem. 
	The second one follows by Rellich-Kondrachov Theorem. For what concerns the last one, it is trivially observed that $\Gamma_{L^\infty}(\mathcal I)\geq \Gamma_{L^1}(\mathcal I)$. In order to show the converse inequality, we note that if $(u_n)\subseteq W^{1,\infty}(\Omega;\mathbb R^d)$ converges to $u$ in $L^1$ and $\liminf_{n\to \infty} \bar I(u_n)=\lim \mathcal I(u_n)<+\infty$ then $(\nabla u_n(x))_n \in C$ for a.e. $x \in \Omega$. Since  (up to a subsequence) the sequence $(u_n)_n$  pointwise converge to $u$,  by  Morrey's inequality and by  Rellich-Kondrachov's Theorem, we get that  the sequence  $(u_n)_n$  uniformly converges to $u$. }

\end{rem}

Now, inspired by the arguments in \cite[Theorem 3.1]{GP}, we prove our result dealing with the relaxation of the functional $F$ in \eqref{supess}.

\begin{proof}[Proof of Theorem \ref{relax1}]
%
Taking into account  \eqref{fslcls}, by \cite[Remark 4.4]{ABP} the functional $W^{1,\infty}(\Omega;\R^d) \ni u \to \overline F(u):=
\supess_{ x \in \Om}f^{ls}(\nabla u(x))$
is $w^*$- lower semicontinuous.
Therefore  we have that

\begin{align}\label{relineq1}
\supess_{x\in \Om}f^{ls}(\nabla u(x))\leq \Gamma_{w*}(F)(u)\leq \Gamma_{w_{seq}^*}(F)(u)
\end{align}
for every $u \in W^{1,\infty}(\Omega;\mathbb R^d)$.
Then it remains  to prove that  $\Gamma_{w_{seq}^*}(F)(u))\leq \overline{F}(u)$ for every $u \in W^{1,\infty}(\Omega;\mathbb R^d)$.

The proof of this inequality  will be developed in several steps.
\begin{itemize}
\item[Step 1. ] First we assume that $f$ satisfies the further hypotheses that 
\beq\label{suplin}f(\xi)\geq \alpha |\xi| 
\eeq for $\alpha >0$ and that   there exists $\bar \xi$ such that $f(\bar \xi)=\min_{\mathbb R^{d \times N}} f$.
Up to a translation argument there is no loss of generality in assuming $\overline \xi=\underline 0$ and $\min_{\mathbb R^{d \times N}}f=0$.
 
 Let $\overline u \in W^{1,\infty}(\Omega;\mathbb R^d)$  and set
\beq
\label{Mlevelsetflslc}
\lambda:= \supess_{ x \in \Omega}f^{ls}(\nabla \overline u(x)).
\eeq

We determine a sequence $(u_{\varepsilon_n}) \subset W^{1,\infty}(\Omega;\mathbb R^d)$ such that $$u_{\varepsilon_n} \overset{\ast}{\rightharpoonup} \overline u \hbox{ in } W^{1,\infty}(\Omega;\mathbb R^d)$$ and  $$\lim_{n\to \infty}\supess_{ x \in \Om} f(\nabla u_{\varepsilon_n} (x)) \leq \la.$$
  
With this aim for fixed  $\ee>0$ let $$C_\varepsilon:=\{\xi \in \mathbb R^{d\times N} : f(\xi)\leq \la+\varepsilon\},$$
denote by $I_{C_\varepsilon}$ the indicator function of $C_\varepsilon$, i.e.,
\begin{equation*}
\label{indicator}
I_{C_\varepsilon}(\xi):=\left\{
\begin{array}{ll}
0 &\hbox{ if }\xi \in C_{\varepsilon},\\
+\infty &\hbox{otherwise.} 
\end{array}\right.
\end{equation*}
Clearly  $\underline 0 \in C_{\varepsilon}$.  Since $\la+\epsilon>\inf f$ we get that $C_\varepsilon$ is convex,
 and  has nonempty interior. Moreover the coercivity of $f$ guarantees that $C_\varepsilon$ is bounded.  Set
\begin{equation*}\label{Cinfty}
C^\infty:=\{\xi \in \mathbb R^{N\times d}: f^{ls}(\xi)\leq \la\}.
\end{equation*}
By \eqref{Mlevelsetflslc},   $\nabla \overline u(x) \in C^\infty$ for a.e. $x \in \Omega$ and, by Proposition \ref{properties},  it holds 
$$
C^\infty =\bigcap_{\varepsilon >0}{\overline C_\varepsilon}.
$$
Then $\nabla \overline u(x) \in {\overline C_\varepsilon}$ for a.e. $x \in \Omega$ and for every $\ee>0$.

For fixed  $\ee>0$ denote by ${\mathcal G}_\varepsilon$ and  $	\tilde{\mathcal G}_\varepsilon$ the unbounded integral functionals defined in $W^{1,\infty}(\Omega;\mathbb R^d)$ with values in $[0,+\infty]$, as 
\begin{equation*}
\label{GM}
{\mathcal G}_\varepsilon(u):=\int_{\Omega}I_{C_\varepsilon}(\nabla u(x))dx,
\end{equation*}
and
\begin{equation*}
\label{tildeGM}
\tilde{\mathcal G}_\varepsilon(u):=\int_{\Omega}I_{\rm int(C_\varepsilon)}(\nabla u(x))dx.
\end{equation*}
Let $\Gamma_{L^1}({\mathcal G}_\varepsilon)$  and $\Gamma_{L^1}(\tilde{\mathcal G}_\varepsilon)$ be their lower semicontinuous envelopes with respect to the  $L^1$-topology.
Since ${\rm int}{(C_{\varepsilon})}\not =\emptyset$, by \cite[Proposition 1.1.5]{CDeA} we have that $\overline{\rm int(C_{\varepsilon})} = {\overline C_\varepsilon}$. Therefore, by Theorem \ref{AHM} and Remark \ref{remAHM}, we get 
\begin{equation}
\label{sc-}
\Gamma_{w_{seq}^*}(\tilde{ \mathcal G}_\varepsilon)(u)=\int_{\Omega}I_{\overline C_\varepsilon}(\nabla u(x))dx,
\end{equation}
for every $u\in W^{1,\infty}(\Omega;\mathbb R^d)$.

On the other hand, since 
\begin{equation*}
\int_{\Omega}I_{\overline C_\varepsilon}(\nabla u(x))dx\leq \mathcal G_{\varepsilon}(u)\leq \tilde{\mathcal G}_\varepsilon(u),
\end{equation*}
 we get that
$\Gamma_{w_{seq}^*}(\tilde{ \mathcal G}_\varepsilon)=\Gamma_{w_{seq}^*}( \mathcal G_\varepsilon).$

We notice that the latter equality and the representation formula \eqref{sc-}  imply that for every $u\in W^{1,\infty}(\Omega;\mathbb R^d)$ $$\Gamma_{w_{seq}^*}( \mathcal G_\varepsilon)(u)=0 \Longleftrightarrow \nabla u(x) \in {\overline C_\varepsilon}\hbox{ for a.e. }x \in \Omega.$$ 
In particular, if  $u\in W^{1,\infty}(\Omega;\mathbb R^d)$ is such that  $\nabla u(x) \in {\overline C_\varepsilon}\hbox{ for a.e. }x \in \Omega$ then  there exists a sequence $(v^{\ee}_k)$ converging weakly* to $u$ in $W^{1,\infty}(\Omega;\mathbb R^d)$ such that 
$$
0=\int_{\Omega}I_{\overline C_\varepsilon}(\nabla u(x))dx = \lim_{k}\int_{\Omega} I_{C_\varepsilon}(\nabla v^{\ee}_k(x))dx.
$$
Thus, by the regularity of $\Omega$, the previous identity implies that  there exists $\overline{k}$ (depending on $\varepsilon$)   such that for every $k\geq \overline{k}$ 
$$
\left\{\begin {array}{ll}
\displaystyle 
\nabla v^{\ee}_k(x) \in C_\varepsilon \hbox{ for a.e. } x \in \Omega,\\
\\
\|u-v^{\ee}_{k}\|_{L^\infty}\leq \varepsilon, \\
\end{array}\right.
$$
which equivalently means that  for every $k \geq \overline{k}$
\beq\label{epsilon}
\left\{\begin {array}{ll}
\displaystyle 
f(\nabla v^{\ee}_k(x))\leq \la+\varepsilon \hbox{ for a.e. } x \in \Omega, \\
\\
\|u-v^{\ee}_{k}\|_{L^\infty}\leq \varepsilon. \\
\end{array}.\right.
\eeq

%
%
%
Now for every $n\in \N$ let  $\varepsilon_n>0$ be such that $\varepsilon_n\to 0$.  Since   $\nabla \bar u(x) \in \overline{ C}_{\varepsilon_n} \hbox{ for a.e. }x \in \Omega$ and for every $n\in \N$, by  applying \eqref{epsilon} with $\ee_n$, we can find two sequences $(k_n)$ strictly increasing and such that $k_n\geq n$, and $(v^{\ee_n}_{k_n})\subseteq W^{1,\infty}(\Omega;\mathbb R^d)$  satisfying 
\beq\nonumber
\left\{\begin {array}{ll}
\displaystyle 
f(\nabla v^{\ee_n}_{k_n}(x))\leq \la+\varepsilon_n\hbox{ for a.e. } x \in \Omega, \\
\\
\|\overline u-v^{\ee_n}_{k_n}\|_{L^\infty}\leq \varepsilon_n. \\
\end{array}.\right.
\eeq
Thus we can conclude that for every $n\in \mathbb N$ and $\ee_n>0$ there exists $v^{\ee_n}_{k_n}$ such that $\|\overline u-v^{\ee_n}_{k_n}\|_{L^\infty}\leq \varepsilon_n$ and 
$$
\supess_{x \in \Omega}f(\nabla v^{\ee_n}_{k_n}) \leq \la+\varepsilon_n.
$$ 

Thanks to the coercivity assumption \eqref{suplin}, we get that $(v^{\ee_n}_{k_n})$   weakly$\ast$ converges to $\bar u$ in $W^{1,\infty}(\Omega;\mathbb R^d)$.  As consequence, it results that
$$
\Gamma_{w^*_{seq}}(F)(u)\leq \lim_{\varepsilon_n \to 0}\supess_{x \in \Omega}f(\nabla v^{\ee_n}_{k_n})\leq \la.
$$
Thus it suffices to define $u_{\varepsilon_n}:=v^{\ee_n}_{k_n}$, to conclude the proof.

\item[Step 2.] Next we remove the coercivity assumption on $f$, just assuming that $f$ admits minimum and $f(\underline 0)=\min_{\mathbb R^{d \times N}}f=0$.
  
For every $n \in \mathbb N$ and every $\xi \in \mathbb R^{d\times N}$, define $f_n$  the level convex function given by 
\begin{equation*}
\label{fn}
f_n(\xi):=\max\big\{f(\xi), \tfrac{1}{n}|\xi|\big\}.
\end{equation*}

Clearly $f_n$ satisfies all the assumptions in Step 1. Thus, defining $f_n^{ls}:=(f_n)^{ls}$, and denoting by $F_n$ the functional defined as $W^{1,\infty}(\Omega;\R^d) \ni u \to F_n(u):=\supess_{ x \in \Omega}f_n(\nabla u(x))$, 
we deduce that
\beq\label{barFnrep}
\Gamma_{w*}(F_n)(u)=\supess_{x \in \Omega}f_n^{ls}(\nabla u),
\eeq
for every $u \in W^{1,\infty}(\Omega;\R^d)$   

Moreover $F_n$ decreasingly converges to $F$ since 
$F_n(u)= \max\{F(u), \frac{1}{n}\|\nabla u\|_{L^\infty}\}$ (see \cite[Remark 3.7]{GP}).
 Thus, by virtue of Proposition \ref{gammaprop} (2)-(3) we can conclude that 
 \begin{equation}
 \label{uno}
 \Gamma_{w*}(F)(u) =\Gamma(w^*)\hbox{-}\lim_{n \to +\infty}F_n(u)=\Gamma(w^*)\hbox{ -}\lim_{n \to +\infty}\Gamma_{w*}(F_n)(u),
 \end{equation}
 for every $u \in W^{1,\infty}(\Omega;\R^d)$.
 Since $f_n (\xi)\leq f(\xi)+\frac 1 n  |\xi|$ for every $\xi \in \R^{d \times N}$, then  
$$
 f_n^{ls} (\xi) -\tfrac{1}{n}  |\xi|\leq f(\xi) 
$$
for every $\xi \in \R^{d \times N}$.
The continuity of $\tfrac{1}{n}|\cdot|$ entails

\begin{equation} \label{tre}
  f_n^{ls}(\xi)-\tfrac{1}{n}|\xi|\leq f^{ls}(\xi). \end{equation}
that yields to 
\begin{eqnarray}\nonumber 
\supess_{x \in\Omega}f_n^{ ls}(\nabla u(x))&\leq&  \supess_{x \in \Omega}(f^{ls}(\nabla u(x))+\tfrac{1}{n}|\nabla u(x)|)\\
\nonumber
&\leq& \supess_{x \in \in\Omega}f^{ls}(\nabla u(x))+ \tfrac{1}{n}\|\nabla u\|_{L^\infty},
\end{eqnarray}
 for every $u \in W^{1,\infty}(\Omega;\R^d)$.
Thanks to \eqref{barFnrep}, we get that 
$$ \Gamma_{w*}(F_n)(u)\leq  \supess_{x \in \in\Omega}f^{ls}(\nabla u(x))+ \tfrac{1}{n}\|\nabla u\|_{L^\infty}.$$

  By the latter inequality,  by \eqref{uno}  and by \eqref{relineq1} 
  we get  that 
\begin{eqnarray}\nonumber 
	\Gamma_{w*}(F)(u)&=&\Gamma(w*)\hbox{-}\lim_{n\to +\infty}\Gamma_{w*}(F_n)(u)\\
	\nonumber 
	&\leq & \lim_{n}(\supess_{x \in \Omega}f^{ls}(\nabla u(x))+ \tfrac{1}{n}\|\nabla u\|_{L^\infty})\\
	\nonumber 
	&=&\supess_{x \in \Omega}f^{ls}(\nabla u(x)=\overline{ F}(u).
	\end{eqnarray}

\item[Step 3.] 
Now we remove the assumption that $f$ admits a minimum.  We assume that $f$ admits a real infimum. The existence of the real infimum of $f$ guarantees that $F$ also admits a real infimum and they coincide.
 By \eqref{inf1}  it results that $$\inf_{W^{1,\infty}(\Omega,\mathbb R^d)}F(u)= \inf_{W^{1,\infty}(\Omega,\mathbb R^d)}\Gamma_{w*}(F)(u)=\inf_{\mathbb R^{d \times N}}f.$$
Thanks to Remark \ref{equiv} there exist two sequences $(\xi_n) \subseteq \R^{d\times N}$ and $(\la_n)_n \searrow \inf_{\mathbb R^{d \times N}}f $  such that

$$
 f(\xi_n)\leq \la_n  \hbox{ and }  \limsup_{\xi \to 0} f(\xi_n+ \xi)\leq  \la_n\quad  \forall n\in \N .$$
 Then $(u_n) \subseteq W^{1,\infty}(\Omega;\mathbb R^d)$  given by $u_n(x):=\xi_n \cdot x$ is an infimizing sequence since 

\begin{equation}
\label{infima}
\lim_{n \to +\infty}F(u_n)=\lim_{n \to +\infty} f(\xi_n)=\inf_{\mathbb R^{d \times N}}f=\inf_{W^{1,\infty}(\Omega,\mathbb R^d)}F.
\end{equation}
Consider, for every $n \in \mathbb N$ and for every $u \in W^{1,\infty}(\Omega;\mathbb R^d)$ the functional
$$G_n(u):=\max\{F(u+u_n), \la_n\} - \la_n=\max\{F(u+u_n)- \la_n,0\}=  \supess_{x \in \Omega}g_n(\nabla u(x)),$$
where $g_n$ is the function defined as $$\R^{d\times N} \ni \xi \to g_n(\xi):= \max\{f(\xi+\xi_n),\la_n\}-\la_n=  \max\{f(\xi+\xi_n)-\la_n,0 \} \geq 0.$$ 
Then $g_n(0)=0=\min_{\R^{d\times N} }g_n.$
Then $G_n$ verifies all the assumptions in Step 2, $g_n$ being in particular level convex.
Thus applying the previous step and Proposition \ref{gammaprop}(5)  
we obtain  that
\begin{equation}\label{eq3} 
\supess_{x \in \Omega} g_n^{ls}(\nabla u(x))=\Gamma_{w*}(G_n)(u)=\max\{\Gamma_{w*}(F)(u+u_n), \la_n\}-\la_n
\end{equation}
On the other by \eqref{DM1eq}, it results,
$$
g_n^{ls}=\max\{f(\cdot + \xi_n)^{ls},\la_n\}-\la_n.
$$
In particular, for every $\xi \in \R^{d \times N}$,
$$
g_n^{ls}(\xi)=\max\{(f(\cdot + \xi_n)^{ls})(\xi),\la_n\}-\la_n=
\max\{f^{ls}(\xi + \xi_n),\la_n\}-\la_n.
$$
From the latter equality, and the first identity in \eqref{eq3}, we deduce that
$$
\Gamma_{w*}(G_n)(u)=\max\{\supess_{ x \in \Omega}f^{ls}(\nabla u+\nabla u_n), \la_n\} -\la_n.
$$

By the last equality in \eqref{eq3} and a translation argument
$$
\max\{\Gamma_{w*}(F)(u), \la_n\}=\Gamma_{w*}(G_n)(u-u_n)+\la_n=\max\{\supess_{ x \in \Omega}f^{ls}(\nabla u(x)),\la_n\}
$$
Taking the limit as $n\to +\infty$ and exploiting \eqref{inf1} and \eqref{infima}, we have
$$\Gamma_{w*}(F)(u)=\lim_{n \to +\infty}\max\{\Gamma_{w*}(F)(u),\la_n\} =\lim_{n \to +\infty}\max\{\supess_{ x \in \Omega}f^{ls}(\nabla u(x)),\la_n\}=\supess_{x \in \Omega}f^{ls}(\nabla u(x)).
$$

\item[Step 4.] Now we treat the general case, where $\inf_{\mathbb R^d}f=-\infty$.
Defining, for every $m \in \mathbb R^+$ the function
$f_m:= \sup\{f, -m\}$ we have that $f_m$ admits a real infimum and falls into the case described in Step 3.
Thus, if for every $u \in W^{1,\infty}(\Omega;\mathbb R^d)$ we define  $F_m(u):= \supess_{x \in \Omega}f_m(\nabla u(x))$,
then it results that (once again, exploiting the level convexity of $f$, and applying \cite[Proposition 2.6]{GP})
$$
\max\{\Gamma_{w*}(F)(u), -m\}=\Gamma_{w*}(F_m)(u):=\supess_{x \in \Omega}f_m^{ls}(\nabla u(x))=\max\{\supess_{x \in \Omega}f^{ls}(\nabla u), -m\}. 
$$
 The proof is concluded by sending $m \to +\infty$.
\end{itemize}
\end{proof}

\begin{rem}\label{relaxedfunctionals} {\rm 
	 In the same spirit of Remark \ref{remAHM}, the assumptions on $\Omega$ guarantee that if $f=f(\xi)$  is  coercive, then the  relaxed functional $\Gamma_{w*}(F)$ coincides on $W^{1,\infty}(\Omega;\mathbb R^d)$ with the lower semicontinuous envelopes of $F$ with respect to the $L^\infty$ and $L^1$ convergences, i.e. $\Gamma_{w*}(F)=\Gamma_{w_{seq}^*}(F)=\Gamma_{L^\infty}(F)=\Gamma_{L^1}(F)$  by the classical embedding theorems.
 
On the other hand Theorem \ref{relax1}, shows  that even without coercivity assumptions  on $f$, it holds
\begin{equation}\label{w*=w*seq}
\Gamma_{w*}(F)=\Gamma_{w_{seq}^*}(F).
\end{equation}
This fact is not surprising  since the level convexity of $F$ entails the validity of Corollary \ref{weakseq}.}
 \end{rem}

 Thanks to Theorem \ref{relax1}, we can deduce that $f^{ls}$  is the strong Morrey quasiconvex ''envelope'' of $f$,  i.e. the greatest strong Morrey quasiconvex minorant of $f$,   provided that  $f$ satisfies $(H)$.
\begin{cor}\label{corfls}  
	Let $f:\R^{d\times N}\to \overline{\R}$ be a level convex Borel function satisfying (H). Then $f^{ls}$ is the greatest strong Morrey quasiconvex function less than or equal to  $f$.
	\end{cor}

\begin{proof}
 Thanks to Remark \ref{fslcls1}, it is sufficient to show that $
h\leq f^{ls}
$ for every  strong Morrey quasiconvex function  such that $ h \leq f$. 
Let  $h:\R^{d\times N}\to \overline{\R}$ be a strong Morrey quasiconvex function such that $ h \leq f$. Then, if $Q=]0,1[^N$, 
 the associated supremal functional 	$W^{1,\infty}(Q;\R^d) \ni u \to S_h(u):=\supess_{ x \in Q} h(\nabla u)$ satisfies   $S_h\leq F$ on $W^{1,\infty}(Q;\R^d)$ and, by \cite[Theorem 2.6]{BJW},  is a  $w^*_{seq}$-lower semicontinuous functional. Then \eqref{w*=w*seq} and Theorem \ref{relax1} imply that
	$$S_h(u)\leq \Gamma_{w*}(F)(u) \hbox{ for every }u \in W^{1,\infty}(Q;\R^d)$$
	and evaluating this latter expression on affine functions $u(x):=\xi \cdot x$, with $\xi \in \mathbb R^{d \times N}$, we get $
h\leq f^{ls}.
$
 \end{proof}
 
%
%
%
%

Theorem \ref{relax1} allows us to extend the relaxation results for indicator functionals provided by Theorem \ref{AHM} to the case where the convex set is unbounded, and not necessarily open, and with no requirement that $\underline 0\in {\rm int}C$.


\begin{cor}\label{unbddint}Let $\Omega$ be a bounded  open set of $\R^N$ with Lipschitz boundary.
	Let $C\subseteq \mathbb R^{d \times N}$ be a convex Borel set with nonempty interior. 
Let ${\mathcal I}, \overline{\mathcal I}:W^{1,\infty}(\Omega;\mathbb R^d)\to[0,+\infty]$ be the functionals defined  by
\eqref{Ical} and \eqref{Ibarcal}. Then
 \beq\label{secondeq}
 \overline{\mathcal I}(u) = \Gamma_{w*} (\mathcal I)(u) = \Gamma_{w^*_{seq}} (\mathcal I)(u) \quad \forall u \in W^{1,\infty}(\Omega;\R^d).
 \eeq

\end{cor}

\begin{proof}

First we show that for every $u\in W^{1,\infty}(\Omega;\mathbb R^d)$ 
\beq\nonumber
	\overline{\mathcal I}(u)\leq \Gamma_{w*}(\mathcal I)(u). 
\eeq
	Without loss of generality we assume that $\overline{\mathcal I}(u)= +\infty$, then there exists a subset $E$  of $\Omega$ with positive measure such that $\nabla  u(x) \not \in \overline{C}$ for every $x \in E$. Now, consider the functional $G: v \in W^{1,\infty}(\Omega;\mathbb R^d)\to \supess_{x \in\Omega}I_C(\nabla v$). Since $g:=I_C$ satisfies $(H)$, then  Theorem \ref{relax1} guarantees that \\  $\Gamma_{w*}(G)(u)= \supess_{x\in\Omega}I_{\overline C}(\nabla u)$. Thus $$\Gamma_{w*}G(u)=+\infty.$$ Thus, by \cite[Proposition 3.3]{DM}, there exists a neighborhood $U$ of $u$ (with respect to the weak* topology of $W^{1,\infty}(\Omega;\mathbb R^d)$) such that $\supess_{x \in \Omega}I_C(w)=+\infty$ for every $w$ in $U$. This, in turn, implies that $\int_\Omega I_C(\nabla w)dx=+\infty$ for every $w\in U$, i.e. $\Gamma_{w*}(\mathcal I)( u)=+\infty$.
	
\noindent Since $\Gamma_{w*}(\mathcal I)\leq \Gamma_{w_{seq}^*}(\mathcal I)$,	in order to conclude the proof, it is sufficient to show that for every $u \in W^{1,\infty}(\Omega;\mathbb R^d)$ 	 \begin{align*}
	\Gamma_{w^*_{seq}}(\mathcal I)(u)\leq \overline{\mathcal I}(u).
	\end{align*}
Without loss of generality, assume that  $u\in W^{1,\infty}(\Omega;\R^d)$ is such that
$ \int_\Omega I_{\overline C}(\nabla u (x))dx =0$, 
i.e. $\nabla u(x) \in \overline C$ for a.e. $x \in \Omega$.
Thus, arguing as above,  we have that  $\Gamma_{w^*}(G)(u)= \Gamma_ {w_{seq}^*}(G)(u)=\supess_{x \in\Omega}I_{\overline C}(\nabla u(x))=0$. 
In particular there exists a  sequence $(u_n)\subseteq  W^{1,\infty}(\Omega;\mathbb R^d)$, such that $u_n \overset{\ast}{\rightharpoonup}  u$ in $W^{1,\infty}(\Omega;\mathbb R^d)$ and
$$0=\supess_{x \in\Omega}I_{\overline C}(\nabla u(x))= \lim_{n\to \infty}\supess_{x \in\Omega}I_C(\nabla u_n (x)).
$$ Consequently there exists $\bar n \in \mathbb N$ such that $\nabla u_n(x) \in C$ for a.e. $x \in \Omega$ and for every $n > \bar n$, and this in turn entails that $\int_\Omega I_{C}( \nabla u_n (x)) dx = 0$ for every $n > \bar n$.
Finally it results
$$ \Gamma_{w_{seq}^*}(\mathcal I)\leq  \lim_{n\to \infty}\int_{\Omega}I_C(\nabla u_n (x))dx=0=\overline{\mathcal I}(u)$$
and this concludes the proof. 
	\end{proof}

\begin{rem}{\rm We underline that the above result has been obtained by a self-contained argument. On the other hand, as observed in Remark \ref{relaxedfunctionals}, the convexity assumption on $C$ allows to obtain the second equality in \eqref{secondeq}  directly by   Corollary \ref{weakseq}.}
	\end{rem}
\section{The $L^p$-approximation}\label{Lpapproximation}
In this section we prove Theorem \ref{curlcase2},
in details, we study 
$\Gamma$-convergence, as $p\to +\infty$,  of the functionals $F_p: C(\bar \Omega;\R^d)\to [0,+\infty)$  given by
\beq\nonumber F_p(u):=\left\{\begin {array}{cl}
\displaystyle \left( \int_{\Om} f^p(x,\nabla u(x))dx \right)^{1/p}
&  \hbox{if } \, u\in W^{1,p}(\Omega;\R^d),\\
+\infty  & \hbox{otherwise}.
\end{array}\right.
\eeq
where $f :\Omega \times \R^{d\times N}$ is $\L^N\otimes \B_{d\times N}$ function satisfying the growth condition \eqref{fcrescita2}.  We show that, as $p\to\infty$,  $(F_p)_{p\geq 1}$ $\Gamma$-converges with respect to  the uniform convergence to the functional $\bar F:C(\bar \Omega;\R^d)\to [0,+\infty)$ given by \eqref{Finfty2}.

With this aim, we first prove the following result, containing an $L^p$- approximation for $f^{lslc}$, that will be useful in the proof of some particular cases of Theorem \ref{curlcase2}. It generalizes \cite[Proposition 2.9]{AP0} where $f$ is assumed to be level convex and lower semicontinuous.
\begin{prop}\label{lev}
	Let \(f:\R^{d\times N}\rightarrow \R\) be  a  Borel function satisfying  
	\beq \label{coe} f(\xi)\geq \alpha |\xi|      
	\eeq
	for a fixed $\alpha>0$ and for every $\xi\in\R^{d\times N}$.  For every $p\ge 1$, let $(f^p)^{**}$ be the lower semicontinuous and convex envelope of $f^p$.
	Then 
	\beq \label{p-approx}\lim_{p\to\infty}((f^p)^{**})^{1/p}(\xi)= f^{lslc}(\xi).
	\eeq
Moreover if $f$ is level convex, then 
	\beq \label{plevconv}\lim_{p\to\infty}((f^p)^{**})^{1/p}(\xi)= \lim_{p\to\infty}(Qf^p)^{1/p}(\xi)= f^{ls}(\xi)\,
	\eeq
	\noindent where $Qf^p:=Q(f^p)$ is the quasiconvex envelope of $f^p$ in \eqref{Qg}.
	

\end{prop}
\begin{proof} 
	Clearly the family $((f^p)^{**})^{1/p})_{p}$ is not decreasing and for every $\xi \in \R^{d\times N}$  and $ p \in [1,+\infty)$ we have that
	$$
	((f^p)^{**})^{1/p}(\xi) \leq f(\xi).
	$$
	Since $((f^p)^{**})^{1/p}$ is lower semicontinuous and level convex, it results that
	\beq\label{CDPest0}
	((f^p)^{**})^{1/p}(\xi) \leq f^{lslc}(\xi) 
	\eeq
	for every  $\xi \in \R^{d\times N},$  and $ p\in [1,+\infty).$
	\noindent Thus the first inequality in \eqref{p-approx} follows as $p \to +\infty$.   
	Moreover, by \cite[Proposition 2.9]{AP0} applied to $ f^{lslc}$, we have that 
	\beq\label{CDPest}  f^{lslc}(\xi)=\lim_{p\to\infty}(((f^{lslc})^p)^{**})^{1/p}(\xi)\leq \lim_{p\to\infty} ((f^p)^{**})^{1/p}(\xi) 
	\eeq
	for every $\xi \in \R^{d\times N}$.
Now we assume that $f$ is level convex. By (\ref{p-approx}) and (\ref{fslcls})  we get  that $$ 	f^{ls}(\xi)= \lim_{p\to\infty} ((f^p)^{**})^{1/p}(\xi)$$

		\noindent We note that for every fixed  $p\geq 1$ the function $(f^p)^{**}$ is  quasiconvex (see Definition \ref{defqcx}). Then $(f^p)^{**}\leq Qf^p\leq   f^p$ that yields to  $((f^p)^{**})^{1 /p}\leq (Qf^p)^{ 1 /p}\leq   f.$  By the continuity of $Qf^p$ (see Remark \ref{Qgcont}), 
		it follows that for every $p\geq 1$  \beq\label{f*qf}((f^p)^{**})^{1/ p}\leq (Qf^p)^{1/ p}\leq    f^{ls}.\eeq
		By applying  H\"older's inequality, it is easy to show that  the family   $(\big ( Qf^p\big)^{1/p} )_{p}$ is not  decreasing. 
		So, by \eqref{f*qf} we get that 
		$$		f^{ls}(\xi)=\lim_{p\to\infty} ((f^p)^{**})^{1/p}(\xi) \leq \lim_{p\to\infty}(Qf^p)^{1/p}(\xi) \leq    f^{ls}(\xi),$$
		for every $\xi \in \R^{d\times N}$, which proves formula \eqref{plevconv}.		
		
	\end{proof}

	\begin{rem}\label{varie}{\rm For every  $\xi\in \R^{d\times N}$ we denote  \beq\label{Qinftyfdef}
			Q_{\infty}f(\xi):=\lim_{p\to\infty}  ( Qf^p\big)^{1/p}(\xi)=\sup_{p\geq 1} \big( Qf^p\big)^{1/p}(\xi).
			\eeq
		\noindent Note that,  if $N=1$ or $d=1$,  then $Qf^p=(f^p)^{**}$  for every $p\geq 1$. Therefore, if  $f$ satisfies \eqref{coe},  by Proposition \ref{lev}, we get that 
	 $Q_\infty f=f^{lslc}$.
}
	 
	\end{rem}
\noindent In \cite{AP}  it has been introduced the class of  functions $f:\R^{d\times N}\to [0,+\infty)$  satisfying   $f=\lim\limits_{p\to\infty}(Q(f^p))^{1/p}$. They  have been referred as curl-$\infty$ quasiconvex.  If $f$ is continuous, in \cite{PNodea} it has been remarked  that any curl-$\infty$ quasiconvex function is strong Morrey quasiconvex (see \eqref{sMqcx}), while it  is currently an open question  whether the converse is true for coercive functions.  The proposition below establishes,  without further assumptions, that the supremum of strong Morrey quasiconvex  functions is itself strong Morrey quasiconvex. In particular, if \(f:\R^{d\times N}\rightarrow [0,+\infty)\) is a  Borel function then, by the very definition \eqref{Qinftyfdef},  we get that $Q_{\infty}f$ is strong Morrey quasiconvex.

	\begin{prop}\label{lev1}
	Let $I$ be a family of indices and let $(f_{\eta})_{\eta\in I}$,  be a family of strong Morrey quasiconvex functions ($f_\eta:\R^{d \times N} \to \R$ for any $\eta\in I$). Then the function $\hat f:=\sup_{\eta} f_{\eta}$ is strong Morrey quasiconvex. In particular, if \(f:\R^{d\times N}\rightarrow [0,+\infty)\) is a  Borel function then  the sequence $( (Qf^p\big)^{1/p})$  converges to the  strong Morrey quasiconvex function $Q_{\infty}f$.
	\end{prop}
\begin{proof} Let $\Omega \subset \R^N$be a bounded open set with Lipschitz boundary as above. For every $\eta\in I$ the functional $$W^{1,\infty}(\Omega;\R^d) \ni u \to F_{\eta}(u):=\supess_{\Omega}f_{\eta}(\nabla u) $$ is sequentially weakly* lower semicontinuous on  $W^{1,\infty}(\Omega;\R^d)$, see \cite[Theorem 2.6]{BJW}. This implies that the functional  $W^{1,\infty}(\Omega;\R^d)\ni u \to \hat F(u):=\sup_{\eta} F_{\eta}(u)$ is  also sequentially weakly* lower semicontinuous.
Since $$\hat F(u)=\sup_{\eta} \supess_{\Omega}f_{\eta}(\nabla u) =\supess_{\Omega} \sup_{\eta} f_{\eta}(\nabla u)=\supess_{\Omega}  \hat f(\nabla u),$$ then, thanks to the necessary condition for sequentially weak* lower semicontinuity of supremal functionals in \cite[Theorem 2.7]{BJW}, we can conclude  that  $\hat f $ is strong Morrey quasiconvex.

In particular, in order to show that $Q_{\infty}f$ is strong Morrey quasiconvex,  it is sufficient to recall that, by  \cite[Proposition 2.4]{BJW}, for any $p\geq 1$ the function $Qf^p$ is strong Morrey quasiconvex.  \end{proof}
		
%
\bigskip

\begin{proof}[Proof of Theorem \ref{curlcase2}] The proof will be achieved in several steps, some of them follow along the lines of \cite[Proof of Theorem 4.2]{AP}. First we prove that for every $p>N$, the relaxed functional $\Gamma_{L^\infty}(F_p)$  admits an integral representation. In the second step we introduce the function $f_\infty$ appearing in \eqref{Finfty2} and obtain the comparison in \eqref{comparison}. Then steps 3. and 4. are devoted to the proof of $\Gamma$-liminf and $\Gamma$-limsup inequalities, respectively.\\
\begin{itemize}
\item[Step 1.]  For every $p\geq 1$ let   $\Gamma_{L^{\infty}} (F_p):C(\bar \Omega;\R^d)\to \overline{\R}$ be the lower semicontinuous envelope  of the functional $F_p$ in \eqref{curl2} with respect to  the   uniform convergence.
Since the family  $(F_p)_{p\geq 1}$ is increasing,  by Proposition \ref{gammaprop}(3)-(4), we have that 
 \begin{equation}\label{disugFtilde2}
 \Gamma(L^{\infty})\hbox{-}\lim_{p\to\infty} F_p = \Gamma(L^{\infty})\hbox{-}\
\lim_{p\to\infty}  \Gamma_{L^{\infty}} (F_p)  =\sup_{p\geq 1} \Gamma_{L^{\infty}} (F_p).
\end{equation}
 Let 
 $G_p:W^{1,p}(\Omega,\R^d)\to  \overline{\R}$ be the functional given by $$ G_p(u):=\left(
\int_\Omega  f^p(x,\nabla u(x))dx \right)^{1/p}.$$ Then, by Theorem \ref{relaxint2}, there exists a Carath\'eodory function $\tilde{ f^p}$, quasiconvex in the second variable, such that 
\beq\label{tilde>Q}
\tilde f^p(x,\xi) \geq Qf^p(x,\xi) \hbox{ for a.e. }x \in \Omega \hbox{ and for every }\xi \in \R^{d\times N},
\eeq and
 $$\Gamma_{w_{seq}} (G_p)(u):=\left(
\int_\Omega  \tilde{f^p}(x,\nabla u(x))dx \right)^{1/p}$$
for every $u \in W^{1,p}(\Omega;\R^d)$.
Now we show that for every $p>N$    $\Gamma_{L^{\infty}} (F_p)$ coincides with the  functional $\phi_p:C(\bar \Omega;\R^d)\to  \overline{\R}$  given by
  $$
\phi_p(u):=\left\{\begin {array}{cl} \displaystyle{\left(
\int_\Omega  \tilde {f^p}(x,\nabla u(x))dx \right)^{1/p}}
 & \hbox{if } \, u\in W^{1,p}(\Omega,\R^d),\\
+\infty  & \hbox{otherwise},
\end{array}\right.
$$
 In order to show that $ \phi_p\leq\Gamma_{L^{\infty}} (F_p)$ we notice that for every $p>1$ the functional $\phi_p$ is  lower semicontinuous  on $C(\overline \Omega;\R^d)$  with respect to  the  uniform convergence. In fact, let  $(u_n)\subseteq C(\overline \Omega,\R^d)$ be such that $u_n\to u$  uniformly and $\liminf\limits_{n\to \infty} \phi_p(u_n)<+\infty$. Without relabelling, take a subsequence such that $\lim\limits_{n\to \infty}\phi_p(u_n)=\liminf\limits_{n\to \infty}\phi_p(u_n)$.
Thanks to the coercivity assumption (\ref{fcrescita2}), we have that the sequence $(u_n)$ is bounded in $W^{1,p}(\Omega,\R^d)$. Therefore,  up to a not relabelled subsequence, $(u_n)$  weakly converges to $u$ in $W^{1,p}(\Omega,\R^d)$. 
Then $$
\phi_p(u)=\Gamma_{w_{seq}} (G_p)(u)\leq \liminf_{n\to \infty} \Gamma_{w_{seq}} (G_p)(u_n)=\liminf_{n\to \infty} \phi_p(u_n).$$ 
 Since $\phi_p\leq F_p$ on $C (\overline \Omega,\R^d)$ and  $\phi_p$ is  lower semicontinuous   with respect to  the   uniform convergence,  we obtain that  \beq\label{unversox}\phi_p(u)\leq \Gamma_{L^{\infty}} (F_p)(u) \quad  \forall \ u\in C(\overline \Omega,\R^d).\eeq
On the other hand,  for every $p>N$ the functional $\Gamma_ {L^{\infty}}(F_p)$ is sequentially  lower semicontinuous on $W^{1,p}(\Omega,\R^d)$  with respect to  the  weak convergence of $W^{1,p}(\Omega,\R^d)$.
In fact, if  $(u_n)\subseteq W^{1,p}(\Omega,\R^d)$ is such that $u_n\rightharpoonup u$ weakly in  $W^{1,p}(\Omega,\R^d)$ then,  thanks to Rellich-Kondrachov Theorem,  we have that $u_n\in C (\overline \Omega,\R^d)$ and $u_n\to u$ uniformly. In particular it follows that $\Gamma_{L^{\infty}} (F_p)(u)\leq \liminf\limits_{n\to \infty} \Gamma_{L^{\infty}} (F_p)(u_n).$\\

\noindent Since $$\Gamma_{L^{\infty}} (F_p)\leq F_p=G_p \quad  \hbox{on } W^{1,p}(\Omega,\R^d)$$ we get that  for every $p>N$ \beq\label{altroversox}\Gamma_{L^{\infty}} (F_p)(u)\leq \Gamma_{w_{seq}} (G_p)(u)=\phi_p(u) \quad \forall \ u\in W^{1,p}(\Omega,\R^d).\eeq
Inequalities \eqref{unversox} and \eqref{altroversox} imply that for every $p>N$
$$\Gamma_{L^{\infty}} (F_p)(u)=\phi_p(u)=\left(
\int_\Omega  \tilde{f^p}(x,\nabla u(x))dx \right)^{1/p} \quad \forall \ u\in W^{1,p}(\Omega,\R^d).$$
%
If  we  show that   
 $\Gamma_{L^{\infty}} (F_p)(u)<+\infty $ if and only if  $u\in  W^{1,p}(\Omega,\R^d)$ then we can conclude that   $\Gamma_{L^{\infty}} (F_p)=\phi_p$ on $C(\overline \Omega,\R^d)$ for every $p>N$.
 In fact  if  $u\in C(\overline \Omega,\R^d)$ is such that $\Gamma_{L^{\infty}} (F_p)(u)<+\infty$ then there exists a sequence
$(u_n)\subseteq C(\overline \Omega,\R^d)$ such that $u_n\to u$ uniformly and $\lim\limits_{n\to \infty} F_p(u_n)=\Gamma_{L^{\infty}} (F_p)(u)<+\infty$.
 Thanks to the coercivity assumption (\ref{fcrescita2}), we have  that the sequence $(u_n)$  is bounded in $W^{1,p}(\Omega,\R^d)$ and, up to a subsequence, weakly converges to $u$ in $W^{1,p}(\Omega,\R^d)$ when $p>1$. In particular $u\in W^{1,p}(\Omega,\R^d)$. The viceversa is trivial.\\
\item[Step 2.]   If $p<q$ then, by applying  H\"older's inequality,  we have that $F_p\leq (\L^N(\Omega))^{1 -\frac p q} F_q.$
In particular $$\Gamma_{L^{\infty}} (F_p)\leq (\L^N(\Omega))^{1 -\frac p q} \Gamma_{L^{\infty}} (F_q).$$
Since  for every $p\geq 1$ $\tilde{ f^p}$ is a  Carath\'edory function, we deduce that 
$\tilde{ f^p}(x,\xi)\leq (\L^N(\Omega))^{1 -\frac p q} \tilde{ f^q}(x,\xi)$ for a.e. $x\in \Omega$ and $\xi\in  \R^{d\times N}$.
Then, set 
\beq\label{finftydef}f_{\infty}(x,\xi):=\sup_{p\geq 1} (\tilde{ f^p})^{ 1/p}(x,\xi),
\eeq we get that  $f_{\infty}$ is $\L^N \otimes \B_{d\times N}$-measurable function, being the supremum of Carath\'eodory functions, and for a.e. $x\in \Omega$ and $\xi\in  \R^{d\times N}$  $$f_{\infty}(x,\xi)=\lim_{p\to\infty}  (\tilde{ f^p})^{1/ p}(x,\xi).$$ Moreover, thanks to  Proposition \ref{lev1},   for a.e. fixed $x\in \Omega$ the function $f_{\infty}(x,\cdot)$ is strong Morrey quasiconvex.  Finally, by \eqref{tilde>Q}, it results that $Qf^p(x,\xi)\leq \tilde{ f^p}(x,\xi)$ for a.e. $x\in \Omega$ and $\xi\in  \R^{d\times N}$. This implies that   $Q_{\infty}(x,\xi)\leq f_{\infty}(x,\xi)$ for a.e. $x\in \Omega$ and $\xi\in  \R^{d\times N}$.
\item[Step 3.] Now we show the $\Gamma$-liminf inequality, that is 
\beq\label{dbound2}
\Gamma(L^\infty)\hbox{-}\lim_{p\to\infty}    F_p \geq \bar F (u) \quad \forall u\in C(\overline \Omega;\R^d).
\eeq
Without loss of generality   we can consider the case when  $u\in C(\overline \Omega;\R^d)$ is
such that $\sup_{p\geq 1} \Gamma_{L^{\infty}} (F_p)(u )<+\infty$. Thanks to the coercivity assumption \eqref{fcrescita2}, we have that $\sup_{p\geq 1} ||u||_{W^{1,p}}=:M<+\infty.$
It follows that $u\in W^{1,\infty}(\Omega;\mathbb R^d)$  and by \eqref{finftydef} and \eqref{fcrescita2} 
$$
\bar F(u)= \supess_{x\in \Om} f_{\infty}(x,\nabla u(x))\leq \beta (1+M)<+\infty.$$
Therefore, for every fixed $\ee>0$,  there exists a measurable set $B_{\ee}\subset \Omega$  such that $\L^N(B_{\ee})>0$ and 
$$ \supess_{x\in\Omega}f_{\infty} (x,\nabla u(x))\leq f_{\infty}(x,\nabla u(x))+\ee
   $$
  for every $x\in B_{\ee}$.
This implies
  $$ \supess_{x\in\Omega} f_{\infty}(x,\nabla u(x))\L^N(B_{\ee})\leq \int_{B_{\ee}}f_{\infty}(x,\nabla u(x))dx+\ee \L^N(B_{\ee})  .
   $$
By Beppo Levi's Theorem, and H\"older's inequality  we obtain 
\begin{eqnarray*} \supess_{x\in\Omega}f_{\infty} (x,\nabla u(x))\L^N(B_{\ee})&\leq &\lim_{p\to\infty}  
\int_{B_{\ee}}  (\tilde{f^p})^{1/ p}(x,\nabla u(x)dx +\ee \L^N(B_{\ee})\\
& \leq& \lim_{p\to\infty} 
\Big(\int_{B_{\ee}}  (\tilde{f^p})(x,\nabla u(x))dx\Big)^{1/ p}\L^N(B_{\ee})^{1-1 /p} +\ee \L^N(B_{\ee}).
\end{eqnarray*} 
It follows that
\begin{eqnarray}\label{Ftildephip} \supess_{x\in\Omega} f_{\infty}(x,\nabla u(x))&\leq &\lim_{p\to\infty}   \Gamma_{L^{\infty}} (F_p) (u)
\L^N(B_{\ee})^{-1/ p} +\ee=  \sup_{p\geq 1} \ \Gamma_{L^{\infty}} (F_p) (u)+\ee\,.
\end{eqnarray} 
By passing to the limit when $\ee\to 0$ and taking into account  \eqref{disugFtilde2}, we get \eqref{dbound2}.

\item[Step 4.]  Now we show the $\Gamma$-limsup inequality,  that is 
\beq\label{upboundx}
\Gamma(L^\infty)\hbox{-}\lim_{p\to\infty}    F_p(u) \leq \bar F (u) \quad \forall u\in C(\overline \Omega;\R^d).
\eeq
Without loss of generality, we consider the case when $u\in W^{1,\infty}(\Omega;\mathbb R^d)$.
Then $$\Gamma_{L^{\infty}} (F_p)(u)= \Bigl(\int_\Omega  \tilde {f^p}(x,\nabla u(x))dx \Bigr)^{1/p}\leq \L^N(\Omega)^{1/ p}\supess_{x\in\Omega}  f_{\infty} (x,\nabla u(x)).$$
In particular, it follows 
\begin{equation} \label{disugFtildex} 
\sup_{p\geq 1} \Gamma_{L^{\infty}} (F_p) (u)\le \lim_{p\to\infty}  \L^N(\Omega)^{1/ p}  \bar F (u)=\bar F (u),
\end{equation}
for every $u \in W^{1,\infty}(\Omega;\R^d)$.
\noindent By \eqref{disugFtildex} and \eqref{disugFtilde2} we get \eqref{upboundx}.

%
%

%
%
%
%
 \end{itemize}
Putting together steps 3. and 4. we conclude the proof.

\end{proof}

\begin{rem}\label{Lpd1}{\rm We note the following facts.
			
		\begin{enumerate}
		\item If $f(x,\cdot)$ is continuous  for a.e. $x\in \Omega$,  Theorem \ref{curlcase2} gives the same representation result for the $\Gamma$-limit shown in \cite{PNodea}.

\item If the supremand  $f(x,\cdot)$ is upper semicontinuous for a.e. $x \in \Omega$, then $f_{\infty}(x,\cdot)=Q_{\infty}f(x,\cdot)$
by Remark \ref{remBut} and \eqref{finftydef}. The same conclusion holds when $f\equiv f(\xi)$.

\noindent In addiction, if $f(x,\cdot)$ is upper semicontinuous and level convex for a.e. $x\in \Omega$, then, in view of \eqref{plevconv}, \eqref{Finfty2} can be specialized, since $$f_{\infty}(x,\cdot)=Q_{\infty}f(x,\cdot)=f^{ls}(x,\cdot) \;\;\hbox{ for a.e. }x \in \Omega.$$ 
The same conclusion holds when $f\equiv f(\xi)$ is level convex.

			\item If $N=1$ or $d=1$, then $Q_\infty f(x,\cdot)=f^{lslc}(x,\cdot)$ for a.e. $x \in \Omega$, (see Remark \ref{varie}). Consequently, by the above arguments, if $f(x,\cdot)$ is upper semicontinuous or $f\equiv f(\xi)$ then 
			we get that $$f_{\infty}(x,\cdot)= Q_\infty f(x,\cdot)= f^{lslc}(x,\cdot)\;\;\hbox{ for a.e. }x \in \Omega.$$ 

			\item In the case when $f_\infty(x,\cdot)= f^{lslc}(x,\cdot)$, the proof of the $\Gamma$-liminf inequality can be simplified.  Indeed
			$f^{ lslc}$ satisfies the assumptions of \cite[Theorem 3.1]{CDPP} and $f^{lslc}\leq f$, then for every $u \in W^{1,\infty}(\Omega;\mathbb R^d)$
			\begin{align}\label{linflslc}
			\supess_{ x \in \Omega}f^{lslc}(x,\nabla u(x))\leq \Gamma(L^{\infty})\,\text{-} \lim_{p\to\infty}\left(\int_{\Omega}(f^{ lslc}(x,\nabla u(x)))^pdx\right)^{1/p}\\
			\nonumber \leq\Gamma(L^{\infty})\text{-}\lim_{p\to\infty} \left(\int_{\Omega}(f^p(x,\nabla u(x)))dx\right)^{1/p}.
			\end{align}
			It is also worth to note that \eqref{linflslc} holds without imposing any growth from above on $f$.

%
\item  We observe that if $f\equiv f(\xi)$, under the weaker  assumption that  $f$ is a Borel function locally bounded and satisfying (up to a constant)  \eqref{coe},
 we can show that the family  of functionals
 $\F_p: C(\bar \Omega;\R^d)\to [0,+\infty]$    given by 
	\beq\label{curl3} \F_p(u):=\left\{\begin {array}{cl}
	\displaystyle \left( \int_{\Om} f^p(\nabla u(x))dx \right)^{1/p}
	&  \hbox{if } \, u\in W^{1,\infty}(\Omega;\R^d),\\
	+\infty  & \hbox{otherwise}
\end{array}\right.
\eeq
$\Gamma(L^\infty)$-converges to the functional 
$\F: C(\bar \Omega;\R^d)\to [0,+\infty]$    given by 
	\beq\nonumber\F(u):=\left\{\begin {array}{cl}
	\displaystyle  \supess_{\Om} Q_{\infty}f(\nabla u(x))
	&  \hbox{if } \, u\in W^{1,\infty}(\Omega;\R^d),\\
	+\infty  & \hbox{otherwise}.
\end{array}\right.
\eeq
Indeed, in this case,    it is sufficient to apply the relaxation result for integral functionals on Sobolev space with respect to the uniform convergence (see \cite[Theorem 9.1]{Dac}) to    get that $$\Gamma_{L^{\infty}} (\F_p)(u)=\left\{\begin {array}{cl}
	\displaystyle \left( \int_{\Om} Qf^p(\nabla u(x))dx \right)^{1/p}
	&  \hbox{if } \, u\in W^{1,\infty}(\Omega;\R^d),\\
	+\infty  & \hbox{otherwise}.
\end{array}\right.
$$
Then the proof develops along the lines of the one of Theorem \ref{curlcase2} and takes into account  the identity $f_{\infty}=Q_{\infty}f$.
			
%
%

\item For the sake of completeness, with the same notations of Theorem \ref{curlcase2}, if $N$ or $d=1$,  one can assume $\Omega$ to be also convex and $f$ to be only Borel measurable  to obtain a representation formula for $\Gamma_{L^1} (G_p)$, see \cite[Theorem 3.10]{DeA}.
			In particular, one obtains, that
			$$\Gamma_{L^1} (G_p)(u)=\left(
			\int_\Omega  (f^p)^{**}(\nabla u(x))dx \right)^{1/p} \quad \forall \ u\in W^{1,p}(\Omega,\R^d).$$  Then, assuming also that $f$ satisfies \eqref{coe}, \eqref{Finfty2} is obtained in the same way as before, relying on the equality $\Gamma_{L^\infty}(F_p)=\phi_p=\Gamma_{L^1}(G_p)$ in $W^{1,p}(\Omega;\R^d)$.
%

		\end{enumerate}
		\color{black}
	}
\end{rem}

\textbf{Acknowledgements} 
 EZ is indebted with Dipartimento di Matematica of University of Ferrara for its kind support and hospitality. 

Both the authors are members of GNAMPA-INdAM, whose support is gratefully acknowledged.

\end{document}